\documentclass[11pt]{amsart}
\usepackage[utf8]{inputenc}
\usepackage{amsmath,amssymb}
\usepackage{mathtools}
\usepackage[T1]{fontenc}
\usepackage{amsxtra,amsmath,amsfonts,amscd, amssymb, mathrsfs, amsthm, mathtools}

\usepackage[all]{xy}
\usepackage{enumerate}
\usepackage[hidelinks]{hyperref}
\usepackage{color}

\numberwithin{paragraph}{section}
\numberwithin{equation}{section}

\renewcommand{\Bbb}{\mathbb{B}}

\newcommand{\Dcal}{\mathcal{D}}

\newcommand{\Ecal}{\mathcal{E}}

\newcommand{\Fcal}{\mathcal{F}}

\newcommand{\Gcal}{\mathcal{G}}

\newcommand{\Hcal}{\mathcal{H}}

\newcommand{\Ical}{\mathcal{I}}

\newcommand{\Jcal}{\mathcal{J}}

\newcommand{\Kcal}{\mathcal{K}}

\newcommand{\Lcal}{\mathscr{L}}

\newcommand{\Mcal}{\mathcal{M}}

\newcommand{\Ncal}{\mathcal{N}}

\newcommand{\Ocal}{\mathcal{O}}

\newcommand{\Scal}{\mathcal{S}}

\newcommand{\Xcal}{\mathscr{X}}

\newcommand{\Zcal}{\mathcal{Z}}

\newcommand{\fL}{{\mathfrak{L}}}

\newcommand{\fV}{{\mathfrak{V}}}

\newcommand{\C}{\mathbb{C}}

\newcommand{\hH}{{\widehat{H}}}

\newcommand{\N}{\mathbb{N}}

\newcommand{\OY}{\mathcal{O}_Y}

\newcommand{\Q}{\mathbb{Q}}

\newcommand{\R}{\mathbb{R}}

\newcommand{\Z}{\mathbb{Z}}

\newcommand{\Lan}{L^{\an}}

\DeclareMathOperator{\red}{{red}}
\DeclareMathOperator{\an}{{an}}
\renewcommand{\and}{\operatorname{and}}

\DeclareMathOperator{\cha}{{char}}

\DeclareMathOperator{\Div}{Div}

\DeclareMathOperator{\im}{{im}}

\DeclareMathOperator{\length}{\ell}

\DeclareMathOperator{\rank}{{rank}}

\DeclareMathOperator{\Spec}{{Spec}}

\DeclareMathOperator{\supp}{{supp}}

\DeclareMathOperator{\tor}{tor}

\DeclareMathOperator{\MA}{MA}

\DeclareMathOperator{\vol}{vol}
\DeclareMathOperator{\cyc}{{cyc}}
\newcommand{\metr}{{\|\phantom{a}\|}}

\newcommand{\st}{ \ \big| \ }

\newcommand{\tors}{\text{tors}}

\renewcommand{\div}{{\operatorname{div}}}

\newcommand{\hhi}{{\widehat{h}^i}}

\newcommand{\Ko}{{K^\circ}}
\newcommand{\Kt}{{\tilde{K}}}

\newcommand{\Xan}{{X^{\an}}}
\newcommand{\Xs}{{\mathscr{X}_s}}
\newcommand{\covol}{{\rm covol}}

\def\KO{{\mathcal O}}

\def\KL{{\mathscr L}}

\def\KX{{{\mathscr X}}}

\def\KU{{\mathscr U}}
\def\an{{\mathrm{an}}}
\def\red{{\mathrm{red}}}
\def\metric{{\|\ \|}}
\def\metr{{\|\ \|}}
\def\modelD{{\mathscr D}}
\def\Pic{{{\rm Pic}\,}}

\def\kcirc{{K^\circ}}
\def\MA{{\rm MA}}

\newcommand{\Hhat}{{\widehat{H}^0}}

\newcommand{\mb}{{\mathbf m}}
\newcommand{\kb}{{\mathbf k}}
\newcommand{\pb}{{\mathbf p}}

\theoremstyle{plain}
\newtheorem{theo}{Theorem}[subsection]

\newtheorem{prop}[theo]{Proposition}

\newtheorem{lemma}[theo]{Lemma}
\newtheorem{cor}[theo]{Corollary}

\newtheorem{theointro}{Theorem}

\theoremstyle{definition}
\newtheorem{defi}[theo]{Definition}

\newtheorem{art}[theo]{}

\theoremstyle{remark}
\newtheorem{rem}[theo]{Remark}

\title[Non-archimedean volumes]{Differentiability of non-archimedean volumes and 
non-archimedean Monge--Amp\`ere equations\\
(With an Appendix by Robert Lazarsfeld)}

\author[J.I.~Burgos Gil]{Jos\'e Ignacio Burgos Gil}
\address{
J. I. Burgos Gil,
Instituto de Ciencias Matem\'aticas (CSIC-UAM-UCM-UCM3), 
Calle Nicol\'as Cabrera 15, Campus de la Universidad 
Aut\'onoma de Madrid, Cantoblanco, 28049 Madrid, Spain}
\email{burgos@icmat.es}

\author[W.~Gubler]{Walter Gubler}
\address{W. Gubler, Mathematik, Universit{\"a}t 
Regensburg, 93040 Regensburg, Germany}
\email{walter.gubler@mathematik.uni-regensburg.de}

\author[P.~Jell]{Philipp Jell}
\address{Philipp Jell,
School of Mathematics,
Georgia Institute of Technology,
686 Cherry Street
Atlanta, GA 30332-0160}
\email{philipp.jell@math.gatech.edu}

\author[K.~Künnemann]{Klaus K{\"u}nnemann}
\address{K. K{\"u}nnemann, Mathematik, Universit{\"a}t 
Regensburg, 93040 Regensburg, Germany}
\email{klaus.kuennemann@mathematik.uni-regensburg.de}

\author[F.~Martin]{Florent Martin}
\address{F. Martin, Mathematik, Universit{\"a}t 
Regensburg, 93040 Regensburg, Germany}
\email{florent.martin@mathematik.uni-regensburg.de}

\thanks{This work was supported by the collaborative research 
center SFB 1085 funded by the Deutsche Forschungsgemeinschaft.
J. I. Burgos was partially supported by MINECO research projects MTM2016-79400-P and by ICMAT Severo Ochoa project SEV-2015-0554.
Robert Lazarsfeld was partially supported by NSF grant DMS-1439285.
 }

\begin{document}

\begin{abstract}
Let $X$ be a normal projective variety over a 
complete  discretely valued field and 
$L$ a line bundle on $X$. 
We denote by $X^\textrm{an}$ the analytification of
$X$ in the sense of Berkovich and equip the analytification 
$L^\an$ of $L$ with a continuous metric $\metr$.
We study non-archimedean volumes, a tool which allows us to control the asymptotic growth of   
small sections  of big powers of $L$.
We prove that the
non-archimedean volume is differentiable at a continuous semipositive metric and that the derivative 
is given by integration with respect to a Monge--Amp\`ere measure.  
Such a differentiability formula had been proposed  
by M.~Kontsevich and Y.~Tschinkel. 
In residue characteristic zero, it 
implies an orthogonality property for non-archimedean
plurisubharmonic functions which allows us to drop an algebraicity  
assumption in a theorem of S.~Boucksom, C.~Favre and M.~Jonsson about the solution to the non-archimedean Monge--Amp\`{e}re equation.
The appendix by R. Lazarsfeld establishes the holomorphic Morse inequalities 
in arbitrary characteristic.

\bigskip

\end{abstract}

\maketitle

\tableofcontents

\section{Introduction}
\label{section introduction}

\subsection{Monge--Amp\`{e}re equations} \label{subsec: MA}
Let $(X,\omega)$ be a compact K\"{a}hler manifold of dimension $n$,
normalized by $\int\omega^{\wedge n}=1$.
For a probability measure $\mu$ on $X$ which is induced  by a smooth volume form, E.~Calabi conjectured  that the {\it Monge-Amp\`{e}re equation} 
$\eta^{\wedge n} = \mu$ 
has a unique  solution by a real smooth $(1,1)$-form $\eta$  in the same de Rham class as $\omega$. 
Uniqueness was proven by E.~Calabi \cite{Calabi-Proceedings, Calabi-incollection} and the
existence of solutions of the Monge--Amp\`ere equation was settled by
S.T.~Yau \cite{yau78}.

Now we consider a field $K$ endowed with a discretely valued complete absolute value. Let $L$ be a line bundle on an $n$-dimensional projective variety $X$ over $K$. For a continuous semipositive metric $\metr$ on $\Lan$, A.~Chambert--Loir has introduced the Monge--Amp\`ere  measure $c_1(L,\metr)^{\wedge n}$ on the analytification $\Xan$ as a Berkovich space (see Section \ref{section-semipositive-metrics} for details). Then $c_1(L,\metr)^{\wedge n}$ is a positive Radon measure of total mass equal to the degree of $X$ with respect to $L$. 
Assume that $X$ is smooth and $L$ is ample.
In the non-archimedean analogue of the Calabi--Yau problem,  there is a positive Radon measure $\mu$ of total mass $\deg_L(X)$ given on $\Xan$ and we ask for a continuous semipositive metric $\metr$ on $\Lan$ with $\mu= c_1(L,\metr)^{\wedge n}$.

Uniqueness of the metric $\metr$ up to scaling was shown by X. Yuan and S.~Zhang
\cite[Cor. 1.2]{yuan-zhang}. 
In \cite{BFJ1,BFJ2}, S.~Boucksom, C.~Favre and M.~Jonsson 
have proved the existence assuming that the residue field $k$ of $K$ has characteristic zero, that $\mu$ is supported on the dual complex of some SNC model of $X$ and that $X$ satisfies the {algebraicity} condition $(\dagger)$. 
The latter means that 
$X$ is defined over the
function field  of a  
curve over $k$ having $K$ as its completion at a closed point.
Condition $(\dagger)$ is essential in their proof,
allowing them to use global methods on the model to prove the
existence of solutions of the non-archimedean Monge--Amp\`ere equation.  
However, this global hypothesis is quite strong as a variety over a field as $\C((t))$ is usually not defined over a function field of a curve over $\C$.

The main motivation of the present work is to remove condition $(\dagger)$, following a strategy    
outlined in  unpublished notes by M.~Kontsevich and Y.~Tschinkel
\cite{kontsevitch-tschinkel}. To this end we need 
some local volumes to replace the global methods used in
\cite{BFJ1,BFJ2}.

\subsection{Volumes of line bundles on algebraic varieties} \label{subsec: vol intro}

Let $k$ be an algebraically closed field and $Y$ a projective variety over $k$ of dimension $n$.
For a line bundle $L$ on $Y$, the {\it volume} 
\begin{equation*}
\label{eq volume intro}
\vol(L) \coloneqq \limsup_m \frac{h^0(Y,L^{\otimes m})}{ m^n /n!}
\end{equation*}
is  in $\R_{\geq 0}$ (see \cite{Laz1}). 
Outside the nef cone, we have Siu's inequality \cite[2.2.47]{Laz1}
in terms of algebraic intersection numbers:  
if $L,M$ are nef, then 
$\vol(L \otimes M^{-1}) \geq L^n\ -nL^{n-1} \cdot M$. 
It is also known that the function 
$\vol$ is differentiable on the big cone \cite{BFJ0}.

For $i \in \N$, A. K\"{u}ronya \cite{Kur06} has introduced
\emph{asymptotic cohomological functions} 
\begin{equation*}
\widehat{h}^i(Y,L) \coloneqq \limsup_m \frac{h^i(Y, L^{\otimes m} )}{ m^n / n!}.
\end{equation*}
In particular $\widehat{h}^0 = \vol$. 
For $L$ nef, and $i>0$, 
one has $\widehat{h}^i(Y,L) =0$ \cite[1.4.40]{Laz1}
and the main difficulty is again to understand $\widehat{h}^i$ outside
of the nef cone. 
For $L$ and $M$ nef line bundles on $Y$, the \emph{asymptotic holomorphic Morse inequalities} 
give   
\begin{equation}
\label{eq morse intro}
\widehat{h}^i(Y,L \otimes M^{-1}) \leq \binom{n}{i} 
L^{n-i} \cdot M^i.
\end{equation}
First, an analytic proof of these inequalities was given 
by J.P.~Demailly \cite{Dem85}. Later F.~Angelini \cite{Ang96} gave an algebraic
proof in characteristic zero. 
For our applications in this paper, we need the volume and 
the asymptotic cohomological functions for projective schemes over an 
arbitrary field $k$. 
In the appendix by R.~Lazarsfeld, there is an algebraic proof of
\eqref{eq morse intro} which works for a projective scheme $Y$ over any field.

We will study cohomological 
functions  in Section \ref{section asymptotic}. 
More precisely we generalize classical results
about the asymptotic behavior of the dimension
of the higher cohomology
of a coherent sheaf $\mathcal F$ on a projective variety 
twisted by a family of divisors $D_1,\ldots, D_m$ 
(see Prop. \ref{prop asymptotic})
and show that the asymptotic  is uniform in $D_1,\ldots, D_m$.
These results might be of independent interest and have been 
used already in \cite{born-nickel-16}.
{In \S \ref{subsection nonreduced}, we consider the more general case of a projective scheme $Y$ over a noetherian ring since we need this for  Sections \ref{section non-archimedean volume} and \ref{section differentiability}. Then $Y$ is allowed to be  non reduced or non irreducible.}

\subsection{Arithmetic Volumes of line bundles} \label{subsec: arvol intro}
A.~Moriwaki \cite{moriwaki-2009} has introduced an arithmetic analogue of
the volume in the setting of Arakelov theory. Let $F$ be a number
field, $Y$ a projective variety over $F$ of dimension $n$ and $L$ a
line bundle on $Y$. For each place $v$ of $F$, let $F_{v}$ be the
completion of $F$ at $v$ and $Y_{v}^{\an}$ the associated analytic
space
(either as a  complex analytic space
or as a Berkovich space). Assume we are given, for each place
$v$, a continuous metric
$\metr_{v}$ on the analytic line bundle 
$L_{v}^{\rm an}$ over $Y_{v}^{\an}$
determined by $L$. We assume also that almost
all metrics 
$\metr_v$ are determined by a  model of $(Y,L)$ over 
some open subset of
$\textrm{Spec}\,\Ocal_{K}$. Write $\overline L=(L,\{\metr_{v}\}_{v})$ for
the line bundle and the metrics. Then the arithmetic volume of
$\overline L$ is defined as
\begin{equation*}
  \widehat{\vol}(\overline L)\coloneqq \limsup_m
  \frac{\log \# \{s\in H^0(Y,L^{\otimes m})\mid  
  \|s\|_{v}^{\otimes m}\le 1\,\
    \forall v\}}{ m^{n+1} /(n+1)!}.
\end{equation*}

A.~Moriwaki \cite{moriwaki-2009} has shown that the arithmetic volume is
continuous. 
H.~Chen \cite{chen08:_posit} has proved that the arithmetic volume 
is in fact a limit as in the classical case. 

The \emph{$\chi$-arithmetic volume} is a variant of the 
arithmetic volume which is also known as the \emph{logarithm
of the sectional capacity}.
Its definition is recalled in Remark \ref{number field case}.
In contrast to the arithmetic volume the $\chi$-arithmetic
volume can also take negative values.
Both volumes agree when $\overline L$ is (arithmetically) nef.
X.~Yuan \cite{Yua08} has proved an analogue of Siu's
inequality for the $\chi$-arithmetic volume 
and used it to prove a very general equidistribution
result.

\subsection{Volumes of balls of bounded sections} \label{subsec: vol ball intro}
Let us now assume that $X$ is a projective variety over 
a local field $K$. 
We also fix a line bundle $L$ on $X$.  
We consider a continuous metric $\metr$ on $\Lan$ and study the asymptotic
behavior of the volume of the sets  
\begin{equation*}
\Hhat(X,L^{\otimes m},\metr^{\otimes m}) \coloneqq \{ s\in \Gamma(X, L^{\otimes m}) \ \big| \  \|s\|_{\sup} \leq 1 \}
\end{equation*}
with respect to a Haar measure $\mu_m$ on $\Gamma(X,L^{\otimes m})$
where 
{$\|s\|_{\sup}=\sup_{p\in X^\an}\|s(p)\|^{\otimes m}$.}
However 
$\mu_m$ is well defined only up to multiplication by a positive constant.
To bypass this ambiguity, one fixes a continuous reference metric $\| \cdot\|_0$
on $\Lan$ and introduces the local volume
\begin{equation}
\label{eq na volumes intro}
\vol(L,\| \ \|, \| \ \|_0) \coloneqq
\limsup_m \frac {n!}  {m^{n+1}} \cdot
\log \left( 
\frac{ 
\mu_m\bigl( \Hhat(X,L^{\otimes m},\metr^{\otimes m}) \bigr)
}{ 
\mu_m\bigl(\Hhat(X,L^{\otimes m},\metr_0^{\otimes m})\bigr)  
} 
\right).
\end{equation}
These \emph{local volumes} will be called
\emph{archimedean} or \emph{non-archimedean} depending on 
the nature of the local ground field $K$. 
If $F$ is a number field,  $K$ is the completion of $F$ at a non-archimedean place $v$ {and $L$ is ample}, then 
we will show in Remark \ref{number field case} that the local volume
at $v$ is a local version of the $\chi$-arithmetic volume obtained by 
choosing fixed metrics at the other places.

{Non-archimedean volumes were introduced}
by M.~Kontsevich and Y.~Tschinkel 
in \cite{kontsevitch-tschinkel}.
Furthermore differentiability for this
local volume was proposed \cite[p.30]{kontsevitch-tschinkel}.

In the archimedean context R.~Berman and S.~Boucksom have introduced
and studied in \cite{BB10} a variant of the  
archimedean volume. 
For an ample line bundle, they introduce an energy functional on the space of continuous metrics. 
They prove  that the
 archimedean volume of two metrics agrees with the relative energy of the two
metrics (see \cite[Thm.~A]{BB10})
and that  
the energy satisfies a differentiability property
(see \cite[Thm.~B]{BB10}). 

{A variant of local volumes has been studied by
H.~Chen and C.~Maclean \cite{chen-maclean2015}.
They work over {$\C$ or over any non-archimedean field} 
and associate with $(L,\| \ \|, \| \ \|_0)$ a sequence of logarithmic
ratios between determinants of sup norms 
on  graded  linear systems associated with 
$L$. {Their main result implies conditions when the limsup in \eqref{eq na volumes intro} is a limit (see Remark \ref{conjecture}).}}

\subsection{Differentiability of non-archimedean volumes}  
\label{subsec: diff intro}
Let us now turn back to the non-archimedean situation and explain the
main results of this paper. 
We fix $K$ a  complete discretely valued field with
discrete valuation ring $K^{\circ}$ and $X$
a normal projective variety over $K$ equipped with  
{a} line bundle $L$.  
In this context, a non-archimedean analogue of a smooth hermitian
metric is an {\it algebraic metric} associated to 
 a $K^\circ$ model  $(\KX,\KL)$ of $(X,L)$. 
 The algebraic metric 
is called {\it semipositive} if $\KL|_{\KX_s}$ is nef. 
{A metric is called a {\it semipositive model metric} if a suitable positive tensor power is a semipositive algebraic metric.}
 We call $\metr$ on $\Lan$  a \emph{continuous semipositive metric} if it is a uniform limit
 of  semipositive model metrics.  
Such metrics were first considered by S. Zhang \cite{zhang-95}.  
A construction of A.~Chambert-Loir
\cite{chambert-loir-2006} gives an  associated Monge--Amp\`ere measure  $c_1(L, \| \ \|)^{\wedge n}$ on  $X^{\an}$ which is important for arithmetic equidistribution theorems. 
For details, we refer to Section \ref{section-semipositive-metrics}.

Given two continuous metrics $\metr_1, \metr_2$ on $L^{\an}$, 
we define 
$\vol(L,\metr_1, \metr_2)$ similarly as in \eqref{eq na volumes intro}. 
However, since fields such as $\C((t))$ are not locally compact, 
we use the length of the virtual $K^\circ$-module 
$\Hhat(X,L,\metr_1)/\Hhat(X,L,\metr_2)$ instead of
the quotient of the Haar measures (for details see
\S \ref{subsection non-archimedean volume}).

In Theorem \ref{cor3 energy} we prove a non-archimedean analogue of \cite[Thm.~A]{BB10}: 
\begin{theointro} \label{theo energy intro}
If $\| \ \|_1, \| \ \|_2$ are two continuous semipositive metrics on $L^{\an}$, 
then 
\begin{equation*}
\vol(L,\| \ \|_1, \| \ \|_2) = \frac{1}{n+1} \sum_{j=0}^n \int_{X^{\an}} -\log 
\frac{\| \ \|_1}{\| \ \|_2} \, c_1(L, \| \ \|_1)^{\wedge (n-j)} \wedge c_1(L,\| \ \|_2)^{\wedge j}.
\end{equation*}
\end{theointro}
{From the proof of this equation} we deduce that for continuous semipositive metrics the limsup
in the definition of $\vol(L,\| \ \|_1, \| \ \|_2) $ is actually a
limit. 
S.~Boucksom and D.~Eriksson  told us that they have a
proof of Theorem \ref{theo energy intro} using different methods\footnote{{See now \cite[Theorem 8.5]{boucksom-eriksson2018} which holds for any non-archimedean field $K$ under the additional assumptions that $X$ is smooth and $L$ is ample.}}.
Our proof is based on a study of non-archimedean volumes and on the 
results of Section \ref{section asymptotic}.

Our main result (following from Theorem \ref{corollary differentiability 3}) is the
\emph{differentiability of the non-archimedean volume}
over any discretely valued complete field $K$:
\begin{theointro}  \label{theo differentiability intro}
Let $\| \ \|$ be a continuous semipositive metric on $L^{\an}$ and $f\colon X^{\an} \to \R$ a continuous function. 
Then if we consider everything fixed except $\varepsilon \in \R $, one has 
\begin{equation} 
\label{eq differentiablity 1intro}
\vol(L,\| \ \|e^{-\varepsilon f}, \| \ \| ) {=} \\
\varepsilon  \int_\Xan f c_1(L,\| \ \|)^{\wedge n} + {o(\varepsilon)}
\end{equation}
for $\varepsilon \to 0$.
Equivalently the function
$t\in \R \mapsto \vol(L, \| \ \|e^{-tf}, \| \ \|)$ is
differentiable at $t=0$ and 
\begin{equation*}
\left. \frac{d}{dt}\right|_{t=0}  \vol(L, \| \ \|e^{-tf}, \| \ \|) =
\int_{X^{\an}} f c_1(L,\| \ \|)^{\wedge n}.
\end{equation*}
\end{theointro}
This formula is the exact non-archimedean analogue of \cite[Thm.~B]{BB10}, 
and was proposed 
by M. Kontsevich and
Y. Tschinkel \cite[\S 7.2]{kontsevitch-tschinkel}. 

Section \ref{section differentiability} is devoted to the proof of 
Theorem \ref{theo differentiability intro}. The proof of Theorem \ref{theo differentiability intro} 
is similar to the proof of Theorem \ref{theo energy intro}, but
additional problems arise from leaving the nef cone.

Our arguments were inspired by the techniques of A. Abbes and T. Bouche
\cite{abbes-bouche} and X. Yuan \cite{Yua08}. 
In fact the differentiability of the non-archimedean volume 
in Theorem \ref{theo differentiability intro}
is related to the differentiability of the 
$\chi$-arithmetic volume 
shown by Yuan (see \cite{Yua08} and \cite[\S 4.4]{chen-lms}) as follows.
{If $K$ is a completion of a number field $F$ at a non-archimedean
place, if $X,L$ are defined over $F$ and if $L$ is ample, then Yuan proves differentiability of the $\chi$-arithmetic volume. Using the relation between the $\chi$-arithmetic
and the non-archimedean volume explained in Remark \ref{number field case}, this implies Theorem  \ref{theo differentiability intro} under the above assumptions on $X$ and $L$.} 
Conversely Theorem \ref{theo differentiability intro} implies the differentiability of the 
$\chi$-arithmetic volume in the direction of a 
non-archimedean metric change.

{The proof of the differentiability of the arithmetic volume in
  the global case can be made in two
  steps. The  first one is to prove only the  inequality ``$\geq$'' in equation
  \eqref{eq differentiablity 1intro} for each place of the global
  field (\cite[Lemma 3.3]{Yua08} and its non-archimedean analogue). In
  the non-archimedean situation this inequality is obtained by
  controlling the size of certain groups of global sections. The second
  step is to prove that the arithmetic volume is log concave (see
  \cite[Theorem B]{yuan-2009}). As explained in \cite[4.1]{chen-lms}
  these two ingredients are enough to prove the differentiability}\footnote{{A referee suggested that a similar strategy might be used in
  the local non-archimedean case. First one proves the inequality ``$\geq$'' in
  equation \eqref{eq differentiablity 1intro} by controlling the size
  of  certain $H^{0}$ groups. Second, using Okounkov 
  bodies  as explained in \cite[proof of Theorem 4.5]{chen-maclean2015},
one writes $\vol(L,\| \ \|_{1}, \| \ \|_{2} )$ as the difference of
two log concave quantities, one depending on $\|\ \|_{1}$ and the other on
$\|\ \|_{2}$. This decomposition  will depend on the
choice of a regular $K$-rational point (whose existence is assumed),  
a system of parameters around that point and a monomial order. 
}}.

{In the {non-archimedean} local case we use a different strategy.  
Instead of proving only the
  inequality ``$\geq$'' in equation \eqref{eq differentiablity
    1intro}, we prove directly the full equality \eqref{eq differentiablity
    1intro}, To this end, instead of controlling only the size of
  certain $H^{0}$-groups, we need to control also the size of certain
  first cohomology groups.  This control is achieved through the use
  of the holomorphic Morse  
inequalities and the results on the asymptotic growth 
of algebraic volumes obtained in 
Section \ref{section asymptotic}.  
}

\subsection{Orthogonality and Monge--Amp\`ere equations} 
\label{subsec: OP and MA intro}
We keep the assumptions on $K$ from \S \ref{subsec: diff intro}.
Although we are able to establish the differentiability of the local
non-archimedean volume in arbitrary characteristic, this is not 
yet enough
to solve the non-archimedean Monge--Amp\`ere equation.  
One  important ingredient which is still missing is
the
existence of the continuous semipositive envelope 
$P(\metr)$ for an arbitrary continuous metric $\metr$
on a line bundle $L^{\rm an}$. Given a
continuous metric $\metr$, one defines its \emph{semipositive envelope}  
$P(\metr)$ as the pointwise infimum 
of all metrics $\metr_1$ on $L^{\rm an}$ such that $\metr_1$ 
is a {semipositive model} 
metric on $L^\an$ with $\metr\leq \metr_1$.
It is a priori not clear that  $P(\metr)$ 
is a continuous
semipositive metric on $\Lan$.

From now on, we assume that the characteristic of the residue field $\Kt$ of $K$ is zero and that $L$ is ample. Then the regularization theorem of S.  
Boucksom, C. Favre and M. Jonsson \cite[Thm.~8.3]{BFJ1} ensures that $P(\metr)$ is a continuous semipositive metric.
Using a local approach to semipositivity as in \cite{gubler-kuennemann2, gubler-martin}, we find  $\hH^0(X,L,\metr) = \hH^0(X,L,P(\metr))$ for any continuous metric $\metr$ and 
its semipositive envelope $P(\metr)$, hence  
\begin{equation} \label{volume vanishing}
 \vol(\metr, P(\metr)) = 0. 
\end{equation}

In Corollary  \ref{cor 4 energy}, we will deduce from \eqref{volume vanishing} that 
 the $\limsup$ in the definition of the 
non-archimedean volume is a $\lim$.
Theorem \ref{theo differentiability intro} and \eqref{volume vanishing} yield the {\it orthogonality property}:
\begin{theointro} 
\label{theorem orthogonality intro}
We assume ${\rm char}(\Kt)=0$. Let $L$ be an ample line bundle on a smooth projective variety $X$ over $K$, let $n \coloneqq \dim(X)$, and 
let $\metr$ be a 
continuous metric on $\Lan$. 
Then  
\begin{displaymath}
\int_{X^{\an}}\log \frac {P(\metr)} {\metr} c_1(L, P(\metr))^{\wedge n} =0.
\end{displaymath} 
\end{theointro}
We show this in Theorem \ref{theo ortho new}. This orthogonality property was proven in \cite[Thm.~A.6]{BFJ2} assuming that $X$ satisfies the {algebraicity} condition $(\dagger)$ mentioned in \S \ref{subsec: MA}.    
It follows from the variational method of S.~Boucksom, C.~Favre and M.~Jonsson that the orthogonality property  yields the existence of 
solutions in the non-archimedean Calabi--Yau problem (see \cite[Thm.~8.2]{BFJ2}) and hence Theorem \ref{theorem orthogonality intro} implies:

\begin{theointro} \label{theo intro monge ampere}
We assume ${\rm char}(\Kt)=0$ and that $L$ is an ample line bundle on the smooth projective variety $X$ over $K$. Let $\mu$ be a positive Radon measure on $\Xan$ with $\mu(\Xan)=\deg_L(X)$ and supported on the dual complex of an SNC model of $X$. Then there is a continuous semipositive metric $\metr$ on $\Lan$ with $c_1(L,\metr)^{\wedge n} = \mu$. 
\end{theointro}

Here, {an {\it SNC model} is a regular} projective variety $\KX$ over the valuation ring $\Ko$ with generic fiber $X$ such that the  special fiber, which is not assumed to be reduced, agrees as a closed subset  with a simple normal crossing divisor $D$ of $\KX$. The {\it dual complex} $\Delta_\KX$ of $\KX$ is defined as the dual complex of $D$ and can be realized as a canonical compact subset of $\Xan$ (see \cite[\S 3]{BFJ1} for details).

Recall that uniqueness up to scaling was proven by X.~Yuan and 
S.~Zhang \cite[Cor.~1.2]{yuan-zhang} without any assumptions on the residue characteristic. For a more general existence result in terms of plurisubharmonic functions,  we refer to Corollary \ref{BFJ corollary 2}.

\subsection{{Acknowledgements}}
We are deeply grateful to {S\'ebastien} Boucksom and Henri Guenancia for numerous discussions and {precious advice}. 
We are very thankful to Robert Lazarsfeld for providing us with Appendix A.
We  thank Olivier Benoist, Steven Dale Cutkovsky, {Charles Favre}, {Mattias Jonsson}, Alex K\"uronya, {Matthias Nickel}, Mihnae Popa, Aur\'elien Rodriguez and Mart\'in Sombra for helpful discussions
and the referees for their useful remarks.

\bigskip

\begin{center}
\textsc{Notation and conventions.}
\end{center}

\smallskip

Let $X$ be a scheme. A \emph{divisor} on $X$ is always a Cartier divisor on $X$.
We denote by ${\rm Div}(X)$ the group of Cartier divisors on $X$ and
put ${\rm Div}(X)_\Q={\rm Div}(X)\otimes_\Z\Q$ and
${\rm Div}(X)_\R={\rm Div}(X)\otimes_\Z\R$.

Let $k$ be a field. A \emph{variety $X$ over $k$} is an integral $k$-scheme $X$
which is separated and of finite type. 
A curve is a variety of dimension one.
For $X$ a variety and $D$ a Cartier divisor on $X$ we will 
sometimes write $h^i(D)$ or $h^i(X, D)$ for $h^i(X, \Ocal_X(D))$. 
We also write $H^i(X,D)$ for $H^i(X,\Ocal_{{X}}(D){)}$.
If $\Fcal$ is a coherent sheaf on a scheme $X$ and $D\in \textrm{Div}(X)$
 we write $\Fcal(D)$ for $\Fcal \otimes_{\Ocal_X} \Ocal_X(D)$.

Let $M$ be a module over a commutative ring $A$ with unit.
Then $\ell_A(M)$ denotes the length of the $A$-module $M$. We write
$\ell(M)$ if $A$ is clear from the context.

Let $X$ be a noetherian scheme over a noetherian base scheme $S$. For an $n$-cycle $Z$ on $X$ with support proper over a zero-dimensional subscheme of $S$ and line bundles 
$L_1,\ldots,L_n$ on $X$, there is an 
intersection number
$L_1\cdots L_n\cdot Z\in \Z$.
A definition of such intersection numbers is given in  \cite[Appendix VI.2]{Kol96}  for coherent sheaves $\Fcal$ instead of $Z$, hence we may apply it for $\Fcal \coloneqq \Ocal_Z$ in case of a prime cycle and we extend it by linearity to all cycles of the above form. These intersection numbers are multilinear and satisfy a projection formula, hence they agree with the usual intersection numbers as given in \cite{Ful98} in case of $S=\Spec(R)$ with $R$ a field or a discrete valuation ring. Indeed, functoriality and multilinearity yields that this can be checked for a prime cycle in projective space over a field and hence it follows easily from \cite[Thm.~2.8]{Kol96}.

If $L_i=\Ocal(D_i)$ for Cartier divisors $D_1,\ldots, D_n$ on $X$, then we set 
\begin{equation} \label{intersection product for CD}
D_1\cdots D_n\cdot Z=\Ocal (D_1)\dots \Ocal(D_n)\cdot Z.
\end{equation}
This is a multilinear and symmetric in $D_1, \dots, D_n$. If $Z$ is the fundamental cycle of $X$, then we simply write $D_1 \cdots D_n$ for the intersection product in \eqref{intersection product for CD}.

If $\{M_1,\dots,M_s\}=\{L_1,\dots,L_n\}$, then we write  $M_1^{n_1} \cdots M_s^{n_s} \cdot Z  \coloneqq L_1\cdots L_n\cdot Z$ if $M_j$ occurs $n_j$-times in the  intersection number. We will always use $M_j^{n_j}$ in this way which should not be mixed up with the tensor power $M^{\otimes n}$ of a line bundle $M$.

 \section{Preliminaries on semipositive metrics, {envelopes} and measures}
\label{section-semipositive-metrics}

The aim of this section is to recall the central notions for our paper following the terminology in \cite{BFJ1,BFJ2}.
In this section,  
let $K$ be a complete discretely valued field with valuation ring
$K^\circ$, uniformizer $\pi$, and residue class field $\tilde K=K^\circ/(\pi)$. 
We normalize the absolute value on $K$ in such a way that 
$-\log|\pi|=1$.

\subsection{Models, analytification and {reduction}}
\label{modelanared}

Let $X$ be a proper variety over $K$. Let $S={\rm Spec}\,K^\circ$.
A \emph{model of $X$} is a proper, flat scheme $\KX$ over $S$
together with a fixed isomorphism $h$ between 
$X$ and the generic fibre $\KX_\eta$ of the $S$-scheme $\KX$.
Usually we read $h$ as an identification.
The special fibre $\KX\otimes_{K^\circ} \tilde K$ of $\KX$ over $S$ is denoted by $\KX_s$.

Let $X$ be a variety over $K$. 
We denote by $X^\an$ the analytification of $X$ over $K$ in the
sense of Berkovich \cite[Thm.~3.4.1]{berkovich-book}. 
The $K$-analytic space $X^\an$ consists of a locally compact Hausdorff
topological space together with a sheaf $\KO_{X^\an}$
of regular analytic functions.
The space $X^\an$ is compact if $X$ is proper over $K$.

Let $X$ be a proper variety over $K$.
For a model $\KX$ of $X$ over $K^\circ$
with special fibre $\KX_s$ 
there is a canonical \emph{reduction map}
$\red\colon X^\an\longrightarrow \KX_s$
which is surjective. 
If the model $\KX$ is normal then
for an irreducible component $V$ of $\KX_s$,
its generic point $\xi_V$  has a unique
preimage $x_V$ in $X^\an$ \cite[Prop. 1.3.3]{BurgosPhilipponSombra} called the
\emph{divisorial point determined by $V$}.

\subsection{Metrics, model metrics and model functions}
\label{metrmodelfunct}

In this subsection, we study metrics on a line bundle $L$ of a proper variety $X$ over $K$.

\begin{art} \label{continuous metrics}
A \emph{continuous metric} $\metric$ on $L^\an$ associates with each section
$s\in \Gamma(U,L)$ on some Zariski open
subset $U$ of $X$ a continuous function
$\|s\|\colon U^\an\rightarrow [0,\infty)$
such that $\|f\cdot s\|=|f|\cdot \|s\|$ holds for
each $f\in \KO_X(U)$.
We further require that $\|s\| > 0$ if $s$ is an invertible section of $L$. 
Given a continuous metric $\metr$ on $L^\an$, we define
\begin{equation} \label{equation H zero hat}
\widehat H^0(X,L,\metr)  \coloneqq \bigl\{s\in H^0(X,L)\,\big|\,\|s(p)\|
\leq 1\mbox{ for all }p\in X^\an \bigr\}.
\end{equation}
{
Observe that $\widehat H^0(X,L,\metr)$ is a free $K^\circ$-module of rank $r \coloneqq \dim_KH^0(X,L)$.
To see this, pick a $K$-basis $s_1,\ldots,s_r$ of $H^0(X,L)$ and remark that for an integer $\alpha >0$ big enough, 
$\pi^\alpha\langle s_1 \ldots s_r\rangle_{K^\circ} \subseteq 
\widehat H^0(X,L,\metr)   \subseteq 
\pi^{-\alpha}\langle s_1\ldots s_r\rangle_{K^\circ}$
as two vector space norms on $H^0(X,L)$ are equivalent
\cite[App.~A Thm.~1]{bosch-lectures}.
}

Given a continuous reference metric $\metr_0$ on $L^\an$, any other continuous
metric on $L^\an$ is of the form $\metr=\metr_0 \,e^{-\varphi}$ for some
$\varphi\in C^0(X^\an)$.
We obtain the class of \emph{singular metrics on $L^\an$}
if we allow arbitrary functions $\varphi\colon X^\an\to \R\cup\{-\infty\}$.
\end{art}

\begin{art} \label{distance of metrics}
The space of continuous metrics on $L^\an$ is a metric space for the distance
\begin{equation}\label{distance-for-metrics}
d\bigl(\metr_1,\metr_2\bigr)=\sup_{X^\an}\,\Bigl|\log\frac{\metr_1}{\metr_2}\Bigr|.
\end{equation} 
Convergence for this distance is called \emph{uniform convergence of metrics on $L^\an$}.
\end{art}

\begin{art} \label{models}
Let $L$ be a line bundle on the proper variety $X$.
A \emph{model of $(X,L)$}
or briefly a \emph{model of $L$} consists of a model
$(\KX,h)$ of $X$ together with a line bundle $\KL$ on
$\KX$ and an isomorphism $h'$ between $L$ and $h^*(\KL|_{\KX_\eta})$.
Usually we read $h'$ as an identification.

Let $(\KX,\KL)$ be a model of $(X,L^{\otimes m})$
for some $m\in \N_{>0}$. 
There is a unique metric $\metric_\KL$
on $L^\an$ over $X^\an$ such that the following holds:
Given a frame $t$ of $\KL$ over some open subset $\KU$ of
$\KX$ and a section $s$ of $L$ over $U=X\cap \KU$ such that
$s^{\otimes m}=ht$ for some regular function $h$ on $U$, we have 
$\|s\|=\sqrt[m]{|h|}$ 
on $U^\an\cap \red^{-1}(\KU_s)$.
{Such a metric on $L^\an$  is called a 
\emph{model metric (determined on $\KX$}).}
A model metric is called \emph{algebraic} if we can choose
$m=1$ in the construction above.
Note that model metrics are continuous.
\end{art}

\begin{lemma} \label{sections and metrics} \label{lemma volume models}
Let $X$ be a normal proper variety over $K$ and 
$\KX$ a normal model of $X$.
{For a   model $\KL$ of $L$ over $\KX$, we have} 
$
\Gamma(\KX,\KL) = \widehat{H}^0(X,L,\metr_\KL)$.
\end{lemma}

\begin{proof}
The inclusion $\subseteq$ is obvious. 
Note that every $s \in \Gamma(X,L)$ extends uniquely to a meromorphic section $\widetilde{s}$ of $\KL$. 
It remains to show that $\| s \|_\KL \leq 1$ yields that $\widetilde{s}$ is a global section of $\KL$. 
Since $\KX$ is normal, it is equivalent to show that the 
Weil divisor associated to $\widetilde{s}$ is effective. 
Let $\xi_i$ be the generic point of the irreducible component 
$E_i$ of the special fiber $\KX_s$. 
The local ring $\Ocal_{\KX,\xi_i}$ is a valuation ring 
and we may normalize the corresponding valuation $v_i$ 
such that it extends the given valuation $v$ on $K$. 
Then the multiplicity of the Weil divisor associated to  
$D \coloneqq \div(\widetilde{s})$ in $E_i$ is equal to 
$v_i(\gamma_i)$, where $\gamma_i$ is a local equation of $D$ in $\xi_i$.
Let $x_i$ be the divisorial point of  $\Xan$ corresponding to $E_i$. 
Then it is clear from our assumptions that 
$v_i(\gamma_i)=-\log|\gamma_i(x_i)| \geq 0$. 
Since the restriction $s$ of $\widetilde s$ to the generic fiber $X$ 
is a global section anyway, this proves that the Weil divisor 
associated with $D$ is effective. 
\end{proof}

\begin{art}  \label{model functions}
Each model metric $\metric$ on  $\KO_{X^\an}$ induces a 
continuous real function
$f=-\log \|1\|$ on $X^\an$.
The space of \emph{model functions}
\[
\modelD(X)=\{f\colon X^\an\rightarrow \R\,|\,
f=-\log \|1\| 
\mbox{ for a model metric }\metric \mbox{ on }\KO_X\}
\]
has a natural structure of a $\Q$-vector space.
We write $\modelD(X)_\R=\modelD(X)\otimes_\Q\R$.
It is shown in \cite[Thm.~7.12]{gubler-crelle} 
that the space of model functions 
$\modelD(X)$ is dense in the space
$C^0(X^\an)$ 
for the topology of uniform convergence.
A model function $f=-\log \|1\|$ on $X^\an$ which comes from an
algebraic metric $\metr$ on $\KO_{X^\an}$ is called
a \emph{$\Z$-model function}.

Let $\KX$ be a model of $X$.
We say that a model function $f=-\log \|1\|$ is 
\emph{determined on $\KX$} if the model metric 
$\metric$ is determined on $\KX$.
Let ${\rm Div}_0(\KX)$ denote the subgroup of ${\rm Div}(\KX)$
of vertical Cartier divisors on the model $\KX$.
Each $D\in {\rm Div}_0(\KX)$ determines a model
$\KO(D)$ of $\KO_X$ and an associated model function
$\varphi_D  \coloneqq -\log\|1\|_{\KO(D)}$.
\end{art}

\begin{prop} \label{effective model fct}
{Let $D$ be a vertical Cartier divisor on the model $\KX$ of $X$. If $D$ is effective, then $\varphi_D \geq 0$. The converse holds if $\KX$ is normal.}
\end{prop}

\begin{proof} 
{If $D$ is an effective Cartier divisor, then it follows easily from the definition of $\metr_{\Ocal(D)}$ that $\varphi_D \geq 0$. Conversely, if $\varphi_D \geq 0$, then the multiplicity formula \eqref{lem:equalityofmult-eq1} in Lemma \ref{lem:equalityofmult} below shows that the Weil divisor associated to $D$ is effective. Since $\KX$ is normal,  $D$ has to be an effective Cartier divisor \cite[Prop. II. 6.3.A]{Hart}.}
\end{proof}

\begin{rem} \label{non-complete generalization}
We note that Lemma \ref{sections and metrics} and hence Proposition \ref{effective model fct} hold also for a  non-complete discretely valued field $F$. 
The proof of Lemma  \ref{sections and metrics} has to be slightly changed: 
Working on the base change $\Xcal' \coloneqq \Xcal \otimes_{F^\circ} \kcirc$, where $K$ is the completion of $F$, and using $\|s\|_\Lcal \leq 1$, 
it follows from \cite[Proposition 6.5]{gubler-crelle} that $\tilde{s}$ induces an effective Weil divisor on $\Xcal'$. 
Since the special fibres of $\Xcal$ and $\Xcal'$ agree, it follows that the Weil divisor on $\Xcal$ associated to $\tilde{s}$ is effective. 
By normality of $\Xcal$, we conclude again that $s \in \Gamma(\Xcal,\Lcal)$.
\end{rem}

\subsection{Closed (1,1)-forms and semipositive metrics}
\label{closed-forms-pos-metrics}

We consider a model  $\KX$ of a proper variety $X$ over $K$.

\begin{art} \label{form curvature etc}
The finite dimensional real vector space space $N^1(\KX/S)$ is defined as the quotient of
$\Pic(\KX)_\R \coloneqq  {\rm Pic}(\KX) \otimes \R$ by the subspace 
generated by classes of line bundles $\KL$ such that $\KL\cdot C=0$ for each closed curve 
$C$ in $\KX_s$. 
An element $\alpha \in N^1(\KX/S)$ is called \emph{nef} if $\alpha \cdot C \geq 0$ 
for all closed curves $C$ in $\KX_s$.  
We call a line bundle $\KL$ on $\KX$ \emph{nef} if the class of $\KL$ in $N^1(\KX/S)$ is nef.
The \textit{space of closed $(1,1)$-forms on $X$} is defined as 
\begin{align} \label{form definition}
\Zcal^{1,1}(X)  \coloneqq \varinjlim N^1(\KX / S),
\end{align}
{where $\KX$ runs over} the isomorphism classes of models of $X$.

Let $L$ be a line bundle on $X$.
Let $\metr$ be a model metric on $L^\an$ which is determined on $\KX$ by
a model $\KL$ of $L^{\otimes m}$.
The class of $m^{-1}\KL$ in $N^1(\KX/S)$ determines a well defined class
$c_1(L,\metr)\in \Zcal^{1,1}(X)$ {called} the \textit{curvature form $c_1(L,\metr)$ of 
$(L, \metr)$}.
\end{art}

\begin{art} \label{Neron Severi}
{We denote by $N^1(X)$ the real vector space $\Pic(X) \otimes \R$ modulo numerical equivalence. A class in $N^1(X)$ is called {\it ample} if it is an  
$\R_{>0}$-linear combination of classes induced by ample line bundles on $X$.} 
The restriction maps 
$N^1(\KX/S) \rightarrow N^1(X),\,[\KL]\mapsto[\KL|_X]$ induce a linear map
$\{\phantom{a}\}\colon \Zcal^{1,1}(X) \longrightarrow  N^1(X), \, \theta \mapsto \{\theta \}$.
\end{art}

\begin{art} \label{semipositivity}
A closed $(1,1)$-form $\theta$ is called {\it semipositive} if it is represented by a nef
element $\theta_\KX \in N^1(\KX/S)$ for some model $\KX$ of $X$. 
We say that a model metric $\metr$ on $L^\an$ for
a line bundle $L$ on $X$ is {\it semipositive} 
if the same holds for the curvature form 
$c_1(L,\metr)$. 
\end{art}

\begin{art} \label{Zhang}
Let $L$ be a line bundle on $X$. 
Following Zhang \cite{zhang-95} we say that a continuous metric $\| \ \|$ on $L^\an$ is 
{\it {continuous} semipositive} if it is a uniform 
limit  of 
semipositive model metrics on $L^\an$.  
\end{art}

\begin{rem}\label{semipositive metric implies nef}
Let $L$ be a line bundle on $X$ which admits a
continuous semipositive metric.
Then the line bundle $L$ is nef 
(use \cite[Lemma 1.2]{BFJ1} or \cite[4.8]{gubler-martin}).
{This implies in particular that the generic fibre 
$L=\KL|_X$ of a nef line bundle $\KL$ on some model $\KX$
of $X$ is nef.}

\end{rem}

\subsection{Chambert-Loir measures and energy}\label{prelim-monge-ampere}

Throughout this subsection $X$ denotes a normal proper $K$-variety of dimension $n$. 

\begin{art} \label{chambert loir measure for models}
Let $\KX$ be a normal model of $X$.
{For line bundles $\KL_1,\ldots,\KL_n$ on the model $\KX$, Chambert-Loir \cite{chambert-loir-2006} introduced}  
the discrete signed measure 
\begin{equation}\label{defmongeamperemodelcaseeq1}
c_1(\KL_1)\wedge\ldots\wedge c_1(\KL_n)  \coloneqq
\sum\limits_{V}
\ell_{\Ocal_{\KX_s,\xi_V}}(\Ocal_{\KX_s,\xi_V})
(\KL_1\cdots\KL_n\cdot V)\,
\delta_{x_V}
\end{equation}
on $X^\an$, where $V$ runs over the irreducible components
of the special fibre $\KX_s$ of our model, 
$\xi_V$ is the generic point of $V$, 
$x_V$ denotes the divisorial point in $X^\an$ determined by $V$, and
$\delta_{x_V}$ is the Dirac measure supported in 
the point $x_V$.

Let $\KL_1,\ldots,\KL_n$ be nef on $\KX$ with $L_i \coloneqq\KL_i|_{X}$. 
Then the measure \eqref{defmongeamperemodelcaseeq1}  
is positive of total mass {$L_1\cdots L_n\cdot X$}. 
\end{art}

\begin{lemma} \label{lem:equalityofmult}
{Let $E$ be a vertical Cartier divisor on a normal model $\KX$ of $X$ with  model function $\varphi_E$. 
For an irreducible component $V$ of $\KX_s$ with divisorial point $x_V \in \Xan$, let $b_V$ (resp.~$c_V$) be the multiplicity of 
$\KX_s$ (resp.~$E$) in $V$. Then we have}
\begin{equation}\label{lem:equalityofmult-eq1}
c_V = \varphi_E(x_V) \cdot b_V.
\end{equation}
{Moreover, for line bundles $\KL_1, \dots, \KL_n$ on $\KX$, we have} 
\begin{equation}\label{lem:equalityofmult-eq2}
\KL_1 \cdots \KL_n \cdot E = \int_{\Xan} \varphi_E c_1(\KL_1)\wedge\ldots\wedge c_1(\KL_n) 
\end{equation}
\end{lemma}

\begin{proof}
Denote by $\xi_V$ the generic point of $V$. 
Since $\KX$ is normal, it is regular in codimension one. 
Thus there exists a local equation $\gamma$ for $V$ at $\xi_V$. 
Then $\gamma^{c_V}$ is a local equation for $E$. 
By \cite[Prop. 1.3.3]{BurgosPhilipponSombra}, 
the seminorm associated with $x_
V$ is precisely the one which 
comes from the valuation of $\Ocal_{{\KX}, \xi_
V}$. 
{For a uniformizer $\pi$ of $\Ko$, we get}
\begin{align*}
1 = v(\pi) = - \log |\gamma^{b_V}(x_V)| | = - b_V \log |\gamma(x_V)|.
\end{align*}
This implies
\begin{align*}
\varphi_E(x_V) = -\log || 1(x_V) ||_{\Ocal(E)} 
= - c_V \log |\gamma(x_V)| = c_V / b_V
\end{align*}
which proves \eqref{lem:equalityofmult-eq1}. 
{From the first part and \eqref{defmongeamperemodelcaseeq1}, we deduce \eqref{lem:equalityofmult-eq2}.} 
\end{proof}

\begin{art} \label{chambert loir measure for metrized line bundles}
{For  {continuous} {semipositive} metrized line bundles
{$(L_1,\metr_1),\ldots,(L_n,\metr_n)$} on $X$ there exists
 a unique 
{positive} Radon measure 
$c_1(L_1,\metr_1)\wedge\ldots\wedge c_1(L_n,\metr_n)$
of total mass
$L_1\cdots L_n\cdot X$ on $X^\an$ with the following properties
(see \cite{chambert-loir-2006, gubler-2007a}):}
\begin{enumerate}
\item[(i)]
{The map 
$((L_1,\metr_1),\ldots,(L_n,\metr_n))\mapsto
c_1({L_1},\metr_1)\wedge\ldots\wedge c_1({L_n},\metr_n)$
is multilinear and symmetric.}
\item[(ii)]
If the metrics on $({L_1},\metr_1),\ldots,({L_n},\metr_n)$
are induced by line bundles $\KL_1,\ldots,\KL_n$ on a model $\KX$ of $X$
then $c_1(L_1,\metr_1)\wedge\ldots\wedge c_1(L_n,\metr_n)$ agrees with \eqref{defmongeamperemodelcaseeq1}.
\item[(iii)]
If each metric $\metric_i$ is a uniform limit 
of continuous semipositive metrics $(\metric_{ij})_{j\in \N}$
on $L_i^\an$, 
then the  measures 
$(c_1(L_{1},\metric_{1j})\wedge\ldots\wedge 
c_1(L_{n},\metric_{nj}))_{j\in\N}$  on $X^\an$
converge weakly to the measure 
$c_1(L_1,\metr_1)\wedge\ldots\wedge c_1(L_n,\metr_n)$.
\item[(iv)]
Given a morphism $f\colon X'\to X$ of normal proper $K$-varieties 
over $K$ of dimension $n$, we have for 
$\overline{L_i}  \coloneqq (L_i,\metr_i)$ {the projection formula}
\[
f_*\bigl(c_1(f^*\overline{L_1})\wedge\ldots\wedge c_1(f^*\overline{L_n})\bigr)\\
=\deg(f)\,c_1(\overline{L_1})\wedge\ldots\wedge c_1(\overline{L_n})
\]
{where $\deg(f)$ is the degree of the finite
function field extension  $K(X')/K(X)$ if
$f$ is dominant and zero otherwise.}
\end{enumerate} 
{We call $c_1(L_1,\metr_1)\wedge\ldots\wedge c_1(L_n,\metr_n)$ the 
\emph{Chambert-Loir measure for  $\overline{L_1},\ldots,\overline{L_n}$}.}
\end{art}

\begin{defi} \label{definition of energy}
{For {continuous} {semipositive} metrics $\metr_1,\metr_2$ on a line bundle $L$ over $X$, the \emph{energy} is} {defined as}
\begin{equation}\label{new-definition-energy}
E(L,\metr_1,\metr_2)  \coloneqq \frac{1}{n+1}\sum_{j=0}^n
\int_{X^{\rm an}}-\log\frac{\metr_1}{\metr_2}\,
c_1(L,\metr_1)^{\wedge j}\wedge c_1(L,\metr_2)^{\wedge (n-j)}\in \R.
\end{equation}
\end{defi}
This energy is denoted
$E_{\theta}(\varphi)$ with $\theta=c_1(L,\metr_1)$ and
$\varphi=-\log \frac{\metr_1}{\metr_2}$ in \cite[Sect. 6]{BFJ2}.

\begin{art} \label{special energy formula}
{If  $\metr_1,\metr_2$ are algebraic metrics induced by models $\KL_1,\KL_2$ of $L$ on a normal model $\KX$ of $X$, 
then we can write $\KL_1=\KL_2(D)$ for some vertical Cartier 
divisor $D$ on {$\KX$} and \eqref{lem:equalityofmult-eq2} yields the explicit formula}
\begin{equation}\label{-algebraic-definition-energy}
E(L,\metr_{\KL_1},\metr_{\KL_2})=\frac{1}{n+1}\sum\limits_{j=0}^n
\KL_1^j \cdot \KL_2^{n-j}\cdot D.
\end{equation}
\end{art}

\subsection{The semipositive envelope}
Let $X$ be a normal projective variety over $K$,
$L$ a  line bundle on $X$ and $\metr$ a continuous metric
on $L^\an$.

\begin{defi}\label{prelim-semipos-envelope-defi}
The \emph{semipositive envelope of the metric $\metr$} is the
singular metric 
\[
P(\metr)  \coloneqq \inf\bigl\{\metr_1\,\big|\,\metr_1 \mbox{ is a semipositive model 
metric on }L^\an\mbox{ with }\metr\leq \metr_1\}
\]
on $L^\an$ with the infimum taken pointwise on $X^\an$.
\end{defi}

\begin{rem} \label{homogenity of envelope}
(i) By definition, we have $P(\metr^{\otimes m})=P(\metr)^{\otimes m}$ for all $m \in \Z$.

(ii) Assume that the semipositive envelope $P(\metr)$ is a continuous metric. Using that the minimum of two semipositive model metrics is a semipositive model metric \cite[3.11, 3.12]{gubler-martin}, 
we see that $P(\metr)$ is the infimum of a decreasing family of semipositive model metrics and hence it follows from Dini's Theorem that $P(\metr)$ is a continuous semipositive metric.
\end{rem}

{For the rest of this subsection we assume
that $\tilde K$ has characteristic zero and that $L$ is an ample line bundle on a smooth projective variety $X$ over $K$.
In \cite{BFJ1}, the envelope was introduced} 
in terms of $\theta$-psh functions. To compare, 
let us fix a model metric $\metr_0$ on $L^{\an}$ for reference and 
consider $\theta  \coloneqq c_1(L,\metr_0)$. 
The function $-\log( P(\metr)/ \metr)$ is the $\theta$-psh envelope 
of the continuous 
function $-\log(\metr/\metr_0)$ on $\Xan$ 
as defined in \cite[Def. 8.1]{BFJ1} and \cite[Thm.~8.3]{BFJ1}  gives the following:

\begin{theo}[Boucksom, Favre, Jonsson]\label{prelim-semipos-envelope}
{Assume $\cha(\tilde K)=0$ and that $L$ is an ample line bundle on a smooth projective variety over $K$. 
Then the} semipositive envelope $P(\metr)$ 
is a continuous semipositive metric on $L^\an$.
\end{theo}

\section{Asymptotic formulas for algebraic volumes }
\label{section asymptotic}

The goal of this section is to study the asymptotics of $h^i(Y,{m_1D_1+\ldots+m_rD_r})$ for fixed divisors $D_1,\ldots,D_r$ on a projective variety $Y$ over any field $k$. 
Our main result is Proposition \ref{prop asymptotic}.  Its consequences from  \S \ref{subsection nonreduced} will be applied in Sections \ref{section non-archimedean volume} and \ref{section differentiability}.   
In these applications, we will need to consider non-reduced projective schemes $Y$ over a non-reduced basis as $R= \Ko / (\pi^\alpha)$ for  a uniformizer $\pi$ of a discrete valuation ring $\Ko$ and a non-zero $\alpha$. 
Note that $R$ is not {necessarily} an algebra over the residue field. 
{Therefore we will develop} much of the theory 
over any  noetherian ring $R$  
in the spirit of the appendix in \cite[\S VI.2]{Kol96}.

Let us recall that the canonical morphism $\Div(Y) \to \Pic(Y)$ is surjective 
if the scheme $Y$ is projective over the noetherian scheme $S= \Spec(R)$ \cite[Cor.~21.3.5]{EGAIV4}.
This means that we can switch freely between the language of Cartier divisors and the language of line bundles. In this section, we have a slight preference to the former.

\subsection{Infinitesimal perturbations}
\label{subsection infinitesimal}

In this {subsection}, let $S=\Spec(R)$ for any noetherian ring $R$ and 
consider a projective scheme $Y$ over $S$. 
We fix a coherent $\Ocal_Y$-module $\Fcal$ on $Y$ with support over a 
{zero-dimensional closed subset}  of $S=\Spec(R)$. 
The dimension of the support of $\Fcal$ is 
denoted by $n$. We note that the cohomology $H^q(Y,\Fcal)$ is an $R$-module of 
finite length and we set 
\[ 
h^q(Y,\Fcal) \coloneqq \length_R\bigl(H^q(Y,\Fcal)\bigr).
\]

\begin{lemma}
\label{lemma very weak Bertini} \label{corollary difference effective}
Let $T$ be a finite subset of {$Y$}, let $D$ be a Cartier divisor on $Y$ 
and let $A$ be an ample divisor on $Y$. 
Then there exists a sufficiently large $m \in \N$ such that the Cartier 
divisors $mA$ and $D+mA$ are linearly equivalent to effective Cartier divisors 
$E$ and $F$, respectively, with the property   
that the supports of $E$ and $F$ are disjoint to $T$.
\end{lemma}

\begin{proof} 
{Recall that regular global sections are 
precisely those global 
sections which correspond to effective Cartier divisors.  In 
\cite[\href{http://stacks.math.columbia.edu/tag/0AYL}{Tag 0AYL}]{stacks-project}, it is explained that a global 
section is regular if and only if it does not vanish in the associated 
points of {$Y$}.} 

{{\it Step 1. There is  $m_0 \geq 0$ such that for any integer 
$m \geq m_0$ there exists a global section $s$ of $\Ocal(D+mA)$ 
with $s(t) \neq 0$ for all $t \in T$.}}

{Since  $Y$ is  noetherian, the closure of any point 
$t \in T$ contains a closed point $t_0$ by 
\cite[\href{https://stacks.math.columbia.edu/tag/02IL}{Tag 02IL}]{stacks-project}.
If $s$ is a global section  of $\Ocal(D+mA)$ with $s(t_0)\neq 0$ 
then $s(t)\neq 0$ and 
hence we can replace $t$ by $t_0$.
So we may  assume that the points in $T$ are closed.}

{We consider $T$ as a reduced closed subscheme $t\colon T\to Y$
By restriction of regular functions to $T$, we get a short exact sequence of coherent $\Ocal_Y$-modules 
\[
0 \longrightarrow \Kcal \longrightarrow 
\Ocal_Y \longrightarrow t_*\Ocal_T \longrightarrow 0.
\]
We twist by $\Ocal(D+mA)$ and consider the associated long 
exact cohomology sequence. Since $A$ is ample,  Serre's vanishing 
theorem \cite[Theorem III 5.2]{Hart} yields a surjection
\begin{equation*}
\label{eq: les T1}
\Gamma(Y, \Ocal(D+mA)) \longrightarrow \Gamma(T, t^* \Ocal(D+mA) ) \longrightarrow 0
\end{equation*}
for $m \gg 0$. Since $T$ is discrete, $t^* \Ocal(D+mA)$ is trivial
and we find a nowhere vanishing section
$s_0 \in \Gamma(T,t^* \Ocal(D+mA))$.
The above surjection 
allows to lift $s_0$ to a section 
$s \in \Gamma(Y, \Ocal(D+mA))$ which does not vanish 
at any $t\in T$. This proves the first step.}

{{\it Step 2. There is  $m_0 \geq 0$ such that for any integer $m \geq m_0$ there exists an effective Cartier divisor $F$ linearly equivalent to $D+mA$ with $\supp(F) \cap T = \emptyset$.}}

{We enlarge first $T$ to include the finitely many associated points of $Y$. Then we apply the first step to get a global section $s$ of $\Ocal(D+mA)$ with $s(t)\neq 0$ for all $t \in T$. As explained at the beginning, such a global section has to be regular and hence the associated effective Cartier divisor does the job in Step 2.}

{Lemma \ref{lemma very weak Bertini} follows by applying Step 2 twice, once for $D=0$ and once for $D$.} 
\end{proof}

It is well known (see \cite[1.2.33]{Laz1} if the base is a field) that for every integer $q$ 
\begin{equation}
\label{eq classical}
h^q(Y, \Fcal(mD)) = O(m^n).
\end{equation}
We need the following easy generalization. We will fix line bundles $M_1, \dots, M_r$ and $P_1, \dots, P_s$ on $Y$. For $\mb=(m_1,\ldots,m_{r}) \in \N^r$ and $\pb=( p_1,\ldots,p_s) \in \N^{s}$, {$r,s\geq 0$}, we set 
\[
\Fcal(\mb,\pb)  \coloneqq \Fcal \otimes M_1^{\otimes m_1} \otimes \dots\otimes 
M_r^{\otimes m_r} \otimes 
P_1^{\otimes p_1} \otimes \dots\otimes P_{s}^{\otimes p_{s}}.
\]

\begin{prop}
\label{lemma perturbation3}
There is constant  $C\in \R$ (depending on the isomorphism classes of  $\Fcal, M_1, \dots, M_r, P_1, \dots, P_s$) such that for all $m_1,\ldots,m_{r}, p_1,\ldots,p_s \in \N \setminus \{0\}$ we have
\begin{equation*}
\left| h^q (Y, \Fcal(\mb,\pb )) -  h^q (Y, \Fcal( \mathbf{0}, \pb)) \right| 
 \leq C\cdot{m}  (m+p)^{n-1}
\end{equation*}
where $m  \coloneqq \sum_{i=1}^r m_i$ and $p  \coloneqq  \sum_{j=1}^s p_j$. 
\end{prop}

\begin{proof}
{We  prove the claim by induction on $n = \dim(\supp(\Fcal))$. If the support is empty, then $\Fcal=0$ and the claim holds for $n=-\infty$. So we may assume $n \geq 0$.} 
As a first step, we will show the existence of a constant $C'$ depending only on the isomorphism classes of $\Fcal, M_1, \dots, M_r$ 
and of a line bundle $L$ such that 
\begin{equation} \label{eq easy 2}
\left| h^q (Y, \Fcal(\mb) \otimes L) -  h^q (Y, \Fcal(\mb)) \right| \leq C' m^{n-1}
\end{equation}
for all $\mb \in (\N \setminus \{0\})^r$ and $\Fcal(\mb)  \coloneqq \Fcal \otimes M_1^{\otimes m_1} \otimes \dots\otimes 
M_r^{\otimes m_r}$. 
By Lemma \ref{lemma very weak Bertini}, there are effective Cartier divisors $E$ and $F$ of $Y$ such that 
$\Ocal(E-F) \simeq L$ and such that the supports of $E$ and $F$ both do not contain 
{a generic point of $\supp(\Fcal)$.}  
This means that the support of $\Fcal(\mb) |_{E}$ has dimension at most $n-1$. The same also holds for the restriction of $\Fcal(\mb,E)  \coloneqq \Fcal(\mb) \otimes \Ocal(E)$ to $E$ and  for the restrictions to $F$. 
Then we have the short exact sequence 
\begin{equation} \label{seq1}
0 \longrightarrow  \Fcal(\mb)  \stackrel{\otimes s_E}{\longrightarrow}   
\Fcal(\mb, E) \longrightarrow \Fcal(\mb, E)|_{E} \longrightarrow 0
\end{equation}
where $s_E$ is the canonical global section of $\Ocal(E)$. 
By induction on $n$, we have 
\begin{equation} \label{indhyp}
h^q (E, \Fcal(\mb, E)|_E) \leq C_{n-1} \cdot m^{n-1}
\end{equation}
for {a} $C_{n-1} \in \R_{\geq 0}$ depending only on the isomorphism classes of 
$\Fcal$,$M_1, \dots, M_{r}$ and $\Ocal(E)$. Using the  long 
exact cohomology sequence associated to \eqref{seq1}, we deduce 
\begin{equation*} \label{ineq 1}
-h^{q-1}(E, \Fcal(\mb,E)|_E) \leq h^q(Y , \Fcal(\mb, E) ) -
h^q( Y, \Fcal(\mb)) 
\leq h^q(E, \Fcal(\mb, E)|_E). 
\end{equation*}
Using {these inequalities} and \eqref{indhyp}, we get 
\begin{equation} \label{ineq 2} 
\left| h^q(Y , \Fcal(\mb, E) ) -
h^q( Y, \Fcal(\mb)) \right| \leq C_{n-1} \cdot m^{n-1}.
\end{equation}
We apply \eqref{ineq 2} to $\Fcal' \coloneqq\Fcal(E-F)$ instead of $\Fcal$ 
and $F$ instead of $E$. We get $C_{n-1}' \in \R_{\geq 0}$ depending only on the isomorphism classes of $\Fcal, M_1, \dots, M_{r}, \Ocal(E)$ and $\Ocal(F)$ such that 
\begin{equation} \label{ineq 3}
\left| h^q(Y , \Fcal'(\mb, F) ) -
h^q( Y, \Fcal'(\mb)) \right| \leq C_{n-1}' \cdot m^{n-1}.
\end{equation}
Using that $\Fcal'(\mb)  \simeq \Fcal(\mb) \otimes L $ and that $\Fcal'(\mb, F) \simeq \Fcal(\mb,E)$, the inequality \eqref{eq easy 2} follows easily from \eqref{ineq 2} and \eqref{ineq 3} with the constant $C'  \coloneqq C_{n-1} + C_{n-1}'$.

To prove  Proposition \ref{lemma perturbation3}, we apply  \eqref{eq easy 2}  for any $\kb \in \N^r$ with $k=\sum_{j=1}^r k_j$ 
to get 
\begin{equation} \label{ineq 6} 
\left| h^q(Y, \Fcal(\kb,\pb) \otimes L) - h^q(Y, \Fcal(\kb,\pb) \right|  \leq C (k+p)^{n-1}.
\end{equation}
for any $L \in \{M_1, \dots, M_r\}$ with  $C \in \R_{\geq 0}$ depending only on the isomorphism classes of $\Fcal, M_1, \dots, M_{r}, P_1, \dots, P_s$. 
The claim follows from  an $m$-fold application of \eqref{ineq 6}.
\end{proof}

\subsection{D\'evissage and non reduced schemes} \label{devissage subsection}
In this subsection, {we work over} $S=\Spec(R)$ for a noetherian ring $R$. 
The goal is to generalize the following classical fact from \cite[1.5]{Deb01} to 
{the situation over} the base scheme $S$.

\begin{lemma}
\label{lemma basic}
Let $Y$ be an $n$-dimensional projective variety over an arbitrary field $k$ and let $q\in \N$. 
Let 
$D_1,\ldots,D_r$ be Cartier divisors and $\Fcal$ a
coherent sheaf on $Y$. 
Then for $m_1,\ldots, m_r\in \N \setminus \{0\}$ and $m=\sum_{i=1}^r m_i$, we have
\[ 
h^q\Bigl(Y, \Fcal\Bigl({\sum_{i=1}^r} m_i D_i\Bigr)\Bigr) 
= \rank(\Fcal) h^q\Bigl(Y,\Ocal_Y\Bigl({\sum_{i=1}^r} m_i D_i \Bigr)\Bigr) + O(m^{n-1}).
\]
where $\rank(\Fcal)$ is the dimension of the $\Ocal_{Y,\xi}$-vector space 
$\Fcal_\xi$ {at} the generic point $\xi$ of $Y$.
\end{lemma}

We need the following d\'{e}vissage result for
coherent sheaves.   
\begin{lemma}
\label{lemma devissage}
For a coherent sheaf $\Fcal$ on a noetherian scheme $Y$, 
there is a filtration
\begin{equation} \label{devissage display}
0 = \Fcal_0 \subset \Fcal_1 \subset
\ldots \subset \Fcal_s = \Fcal
\end{equation}
by coherent subsheaves, closed integral subschemes
$\iota_j\colon 
Z_j \hookrightarrow Y$ and coherent sheaves of ideals $\Ical_j \subset 
\Ocal_{Z_j}$ with 
{$\supp(\Ical_j)=Z_j$} and 
$\Fcal_j/\Fcal_{j - 1}\simeq \iota_{j, *}( \Ical_j)$ for $j=1,\dots, s$.
\end{lemma}

\begin{proof}
This can be found in \cite[\href{http://stacks.math.columbia.edu/tag/01YC}{Tag 01YC}]{stacks-project} except the precise statement for the support of the 
{$\Ical_j$}. 
The latter follows 
immediately from the argument in {\it loc.~cit.}.
\end{proof}

We have the following generalization of Lemma  \ref{lemma basic}. 

\begin{lemma}
\label{lemma bound cohomology non reduced}
Let $Y$ be a projective scheme over $S$ and let $\Fcal$ be a coherent sheaf on $Y$ with support over a zero dimensional subscheme of $S$. 
We denote by $\{E_i\}_{i\in I}$ the set of irreducible components of
$\supp(\Fcal)$ of maximal dimension $n  \coloneqq \dim(\supp(\Fcal))$. 
Let $D_1,\ldots D_r$ be some Cartier divisors and $q\in \N$. 
Then for $m_1, \ldots, m_r \in \N \setminus \{0\}$ we have 
\begin{equation} \label{equation bound cohomology non reduced}
h^q \Bigl( Y,  \Fcal\Bigl({\sum_{j=1}^r} m_j D_j \Bigr) \Bigr) \leq 
{\sum_{i\in I}} \length_{\Ocal_{Y,{\xi_i}}}(\Fcal_{\xi_i}) 
h^q\Bigl(E_i, \Ocal_Y\Bigl( {\sum_{j=1}^r} m_j D_j\Bigr)|_{E_i} {\Bigr)} + O(m^{n-1}),
\end{equation} 
where $m={\sum_{j=1}^r} m_j$ and where $\xi_i$ is the generic point of $E_i$.
\end{lemma}

\begin{proof}
We proceed by induction on the length $s$ of a d\'{e}vissage of $\Fcal$ as in 
\eqref{devissage display}. The case $s=0$ means that $\Fcal =0$ and the claim is 
obvious.  So we may assume that $s \geq 1$. The corresponding  d\'evissage 
\eqref{devissage display} leads to  the short exact sequence 
\[ 0 \longrightarrow \Gcal\Bigl(\sum_{j = 1}^r m_j D_j \Bigr) \longrightarrow 
\Fcal\Bigl(\sum_{j = 1}^r m_j D_j \Bigr) \longrightarrow \Hcal\bigl(\sum_{j = 1}^r m_j D_j \Bigr) 
\longrightarrow 0 
\]
for $\Gcal  \coloneqq \Fcal_{s-1}$ and $\Hcal  \coloneqq \Fcal/ \Fcal_{s-1}$. The  long exact sequence in cohomology yields
\begin{align} \label{eq les1}
h^q\Bigl(Y, \Fcal\Bigl(\sum_{j = 1}^r m_j D_j \Bigr)\Bigr) \leq  
h^q\Bigl(Y, \Gcal\Bigl(\sum_{j = 1}^r m_j D_j \Bigr)\Bigr) +  
h^q\Bigl(Y, \Hcal\Bigl(\sum_{j = 1}^r m_j D_j \Bigr)\Bigr).
\end{align}
 By definition of the d\'evissage, 
 $\Hcal \simeq \varphi_*(\Ical)$ where 
 $\varphi\colon Z \to Y$ is an integral closed subscheme of $Y$ and $\Ical \subset \Ocal_Z$ is a coherent sheaf of ideals with $\supp(\Ical)=Z$. 
 By  projection formula \cite[Exercise II.5.1 (d)]{Hart} and 
by \cite[III 2.10]{Hart}, we deduce 
\begin{equation} \label{eq push forward}
H^q \Bigl(Y,  \Hcal\Bigl(\sum_{j = 1}^r m_j D_j \Bigr)\Bigr) 
\simeq    
H^q\Bigl(Z, \varphi^* \Bigl(\Ocal_Y\Bigl(\sum_{j = 1}^r m_j D_j \Bigr)\Bigr) \otimes \Ical\Bigr).
\end{equation}

\emph{Case 1.} 
If $\dim(Z)<n$, then 
$h^q(Y, \Hcal(\sum_{j = 1}^r m_j D_j ))  = O(m^{n-1})$ by Proposition \ref{lemma perturbation3},  
hence \eqref{eq les1} yields 
\begin{equation} \label{bound above} 
h^q\Bigl(Y, \Fcal\Bigl(\sum_{j = 1}^r m_j D_j \Bigr)\Bigr) 
\leq h^q\Bigl(Y,  \Gcal\Bigl(\sum_{j = 1}^r m_j D_j \Bigr)\Bigr) + O(m^{n-1}).
\end{equation}
Since $\Hcal$ is the push forward of $\Ical$ from $Z$, the assumption in Case 1 yields $\Hcal_{\xi_i} = 0$ for all $i \in
I$. 
Since the length is additive, we deduce
$\length_{\Ocal_{Y,{\xi_i}}}(\Fcal_{\xi_i}) = \length_{\Ocal_{Y,{\xi_i}}}({\Gcal}_{\xi_i})$ for all $i\in I$. 
Hence the result follows from \eqref{bound above} by the induction hypothesis applied to $\Gcal$.

\emph{Case 2.}
If $\dim Z = n$, then $Z = E_{i_0}$ for some $i_0\in I$. Then the stalk $I_\xi$ at the generic point $\xi$ of $Z$ is a non-zero ideal in the field $\Ocal_{Z,\xi}$ and hence equal to this field. Since $\xi$ is in the support of $\Fcal$, it is lying over a closed point $\eta$ in the base scheme $S$ and hence $Z$ may be viewed as a variety over the residue field of $\eta$. So we may apply 
 Lemma \ref{lemma basic} to the right hand side of \eqref{eq push forward} with $\rank(\Ical)=1$ to get 
\begin{equation} \label{bound in case 2}
h^q\Bigl(Y, \Gcal\Bigl(\sum_{j = 1}^r m_j D_j \Bigr)\Bigr)
=h^q\Bigl(E_{i_0}, \Bigl(\Ocal_Y\Bigl(\sum_{j = 1}^r m_j D_j \Bigr)\Bigr)\Big|_{E_{i_0}}
\Bigr) + O(m^{n-1}).
\end{equation}
Using the additivity of the length, we have 
$
\length_{\Ocal_{Y,{\xi_i}}}(\Fcal_{\xi_i})  = 
\length_{\Ocal_{Y,{\xi_i}}} ({\Gcal}_{\xi_i}) $ for 
$ i \neq i_0 $ and 
$\length_{\Ocal_{Y,{\xi}}}
(\Fcal_{\xi}) 
= 
\length_{\Ocal_{Y,{\xi}}}
({\Gcal}_{\xi})+1$. Hence 
the result follows from \eqref{eq les1} and  \eqref{bound in case 2} using the induction hypothesis applied to $\Gcal$.
\end{proof}

\subsection{Volumes and asymptotic cohomological functions}
\label{sec:higher-volumes}

In this subsection, we assume that $Y$ is a projective variety  over a field 
$k$. We will recall the volume of a Cartier divisor
and its higher cohomological analogues.  
We fix $D$ a Cartier divisor on $Y$. 

\begin{art} \label{classical volume}
The {\it volume} of $D$ or of the corresponding line bundle $L=\Ocal(D)$ is defined by
\[ \vol(D) \coloneqq \vol(L) \coloneqq \limsup_m \frac{h^0(Y,\Ocal_Y(mD))}{m^n / n!}.\]
Since $h^0(Y,\Ocal_Y(mD)) = O(m^n)$, one gets easily that $\vol(D) \in \R_{\geq 0}$.
Actually the $\limsup$ is a $\lim$. 
This  follows from  Fujita's approximation theorem when $k$ is algebraically closed 
(cf.~\cite[11.4.7]{Laz2} {for characteristic zero
and use \cite{takagi2007} in characteristic $p>0$}).
For arbitrary fields, we refer to \cite[Thm.~8.1]{Cut15}.
\end{art}

\begin{rem}
\label{remark volume nef}
If $D$ is nef, then $\vol(D) = D^n$ (cf.~\cite[Cor.~1.4.41]{Laz1}).
\end{rem}

Alex K\"uronya has introduced and studied the following higher volume-type
invariants in \cite{Kur06} called \emph{asymptotic cohomological
functions}.  

\begin{defi} \label{Kueronya} 
For $0\leq i \leq n$, the  \emph{asymptotic cohomological
  function} $\hhi(Y,D)$ is defined by 
\begin{equation}
  \label{eq:8}
  \hhi(Y,D) \coloneqq \limsup_m \frac{h^i(Y,\Ocal_Y(mD))}{m^n / n!}.
\end{equation}
\end{defi}
For $i=0$, we get the volume. For $i>0$, it seems to be unknown if $\limsup$ is a limit.
In case $k=\C$, K\"uronya showed that $ \hhi(Y,D)$ is homogeneous in $D$ and extends uniquely to a continuous homogeneous function $N^1(Y) \to \R_{\geq 0}$. In fact, the arguments work for 
every algebraically closed base field $k$. We will prove in \S \ref{subsection volumes} a weaker continuity property which holds over any field $k$.

\subsection{Asymptotic cohomological functions for real divisors}
\label{subsection volumes}
\label{subsection rounds}

In this subsection, we assume that $Y$ is an $n$-dimensional projective scheme  over
a field $k$. As promised in \S \ref{sec:higher-volumes}, we will extend K\"uronya's asymptotic cohomological functions  
to $\Div_\R(Y)  \coloneqq \Div(Y) \otimes_\Z \R$ and we will characterize them by homogenity and continuity. Note that K\"uronya proved stronger results in the special case of a projective variety over an algebraically closed field (see  \ref{Kueronya}).

\begin{defi}
\label{defi rounds}
Let $D\in \Div_\R(Y)$. Then we have 
$ D = \sum_{i=1}^r a_i D_i$
for suitable $a_i \in \R$ and $D_i \in \Div(Y)$. We call this a \emph{decomposition} $\Dcal$ of $D$. 
We define the \emph{round-up  of $D$} with respect to $\Dcal$ to be 
\[ \lceil D \rceil_\Dcal \coloneqq \sum_{i=1}^r \lceil a_i \rceil D_i \in \Div(Y)\]
and for $q\in \N$ we set 
$h^q(D)_\Dcal \coloneqq h^q(Y, \Ocal_Y( \lceil D \rceil_\Dcal))$.
\end{defi}

\begin{rem}
\label{rem rounds}
{The above definitions indeed depend on the choice of a given decomposition $\Dcal$.
Similar methods are used in  \cite[Thm.~3.5 (i)]{Ful15}.
One can also define canonical round-downs and round-ups for $\R$-Weil divisors \cite[section 9.1]{Laz2}.}
\end{rem}

\begin{lemma}
\label{lemma lattice}
Let $V$ be a finitely generated $\Z$-module and let $x \in  V\otimes_\Z \R$. 
We consider two decompositions 
$x= \sum_{i=1}^p x_i v_i= \sum_{j=1}^q y_j w_j$ 
with $x_i,y_j \in \R$ and $v_i,w_j\in V$. Then  the set 
$ \Scal  \coloneqq 
\{ { \sum_{i=1}^p} \lceil mx_i\rceil v_i - {\sum_{j=1}^q} \lceil my_j\rceil w_j \st m\in \Z \}
$  
is finite. 
\end{lemma}

\begin{proof}
Let us put a euclidean norm $\| \ \|$ on $V_\R \coloneqq V\otimes_\Z \R$. 
For all $m\in \N$, we have   
\[ 
\Bigl\| \Bigl( \sum_{i =1}^p \lceil mx_i\rceil v_i \Bigr) 
-mx \Bigr\| \leq K_1 \coloneqq  \sum_{i = 1}^p \| v_i\| .\]
Similarly, there exists $K_2 \in \R$ for the second decomposition and hence we get 
\[ 
\Bigl\| \Bigl( \sum_{i =1}^p \lceil mx_i\rceil v_i \Bigr) - 
\Bigl( \sum_{j = 1}^q \lceil my_j\rceil w_j\Bigr) \Bigr\| \leq K
\]
for $K\coloneqq K_1 + K_2$. On the other hand  $( \sum_{i =1}^p \lceil mx_i\rceil v_i ) - ( \sum_{j = 1}^q \lceil my_j\rceil w_j) \in V$. 
Since a given ball in $V_\R$ contains only finitely many points in the lattice $\im (V \to V_\R)$, we deduce that the image of $\Scal$ in $V_\R$ is finite. 
The claim follows from the fact that the kernel of the map $V \to V_\R$ is the group of torsion elements which is finite as  $V$ is finitely generated. 
\end{proof}

In the following, we will use linear equivalence $D \sim E$ for real divisors $D,E \in \Div(Y)_\R$ meaning that $D,E$ have the same image in ${\rm Pic}(Y)\otimes_\Z \R$. 

\begin{lemma}
\label{lemma comparison rounds}
Let $D, E\in \Div(Y)_\R$ be real Cartier divisors  with decompositions $\Dcal$ and $\Ecal$.
If $D \sim E$, then there exists $C >0$ such that for all $m, q\in \N$ 
\[ \left| h^q(mD)_\Dcal - h^q(mD)_\Ecal \right| \leq C m^{n-1}. \]
\end{lemma}

\begin{proof}
Let $D= \sum_{i=1}^r a_i D_i$ be the decomposition $\Dcal$ and let $E=\sum_{j=1}^s b_j E_j$ be the decomposition $\Ecal$. 
The images of  $D_1, \dots, D_r, E_1, \dots, E_s$ in $\Pic(Y)$ generate a subgroup $V$. Let $\pi\colon \Div(Y) \to \Pic(Y)$ be 
the canonical homomorphism.   
Using 
$\sum_{i =1}^r a_i \pi(D_i) = \sum_{j =1}^s b_j \pi(E_j)$ in $V_\R$ and
Lemma \ref{lemma lattice},  
$\Scal \coloneqq \biggl\{ \sum \limits_{i=1}^r \lceil ma_i\rceil \pi(D_i) - \sum \limits_{j = 1}^s \lceil mb_j \rceil \pi(E_j) \bigg| m\in \N \biggr\}$ is a finite subset of $\Pic(Y)$.
We fix representatives $G \in \Div(Y)$ of the elements in $\Scal$. Then \eqref{eq easy 2} yields  a constant $C_G$ such that for all $m\in \N$, 
\[ 
\biggl| h^q\Bigl(Y, \Ocal_Y\Bigl(\sum_{j=1}^s \lceil mb_j\rceil E_j + G \Bigr) \Bigr)
- h^q\Bigl(Y, \Ocal_Y\Bigl(\sum_{j = 1}^s \lceil mb_j\rceil E_j\Bigr)\Bigr) \biggr| \leq C_G \Bigl(1+\sum_{j = 1}^s \lceil mb_j \rceil\Bigr)^{n-1}.
\]
Using $h^q(mD)_\Dcal = h^q(Y, \Ocal_Y(\sum_{j =1}^s \lceil mb_j\rceil E_j + G ){)}$ for a suitable representative $G$ and finiteness of $\Scal$, we easily deduce the claim.
\end{proof}

\begin{rem}
\label{rem asymptotic decomposition}

We are interested in the asymptotics of  $h^q(m_1D_1+ \dots + m_r D_r)_\Dcal$ for real divisors $D_1,\dots, D_r$ with respect to decompositions $\Dcal_k$ of $D_k$ and $\Dcal  \coloneqq \coprod_{k=1}^r \Dcal_k$. An obvious generalization of Lemma \ref{lemma comparison rounds} shows that this function depends only on the linear equivalence classes of $D_1, \dots, D_r$ and is independent of the choice of the decompositions $\Dcal_k$ up to an error term of the form $O(m^{n-1})$ for $m  \coloneqq \sum_{k=1}^r m_k$. We use the notation $h^q(m_1D_1+ \dots + m_r D_r)$ 
which is a well defined function in $(m_1, \dots, m_r)$  up to $O(m^{n-1})$.
\end{rem}

\begin{defi} \label{Kueronya for real divisors}
For $D \in \Div(Y)_\R$ and $0 \leq q \leq n$, we define
\[
\widehat{h}^q(Y,D)  \coloneqq \limsup_m \frac{ h^q(Y,m D )}{m^n/n!}.
\]
\end{defi}

By \ref{rem asymptotic decomposition}, the value of $\widehat{h}^q(Y,D)$ depends only on the linear equivalence class 
of $D$ and is independent of the decomposition chosen to calculate $h^q(Y, mD)$.

\begin{lemma}
\label{lemma perturbation real}
Fix $D_1\sim D_1',\ldots, D_r\sim D_r',E_1,\ldots,E_s \in \Div(Y)_\R$ 
and $q\in \N$.
There exists $C\in \R$ (depending on the linear equivalence classes of 
$D_1,\ldots,D_{r}, E_1,\ldots,E_s$) such that for all  
$m_1,\ldots,m_{r}, p_1,\ldots,p_s \in \R_{\geq 0}$ and for 
$m={\sum_{i=1}^r} m_i$ and $p={\sum_{j=1}^s} p_j$, we have 
\begin{equation}
\label{eq easy 4}
\biggl| h^q \Bigl(Y, \sum_{i=1}^{r} m_i D_i + \sum_{j=1}^s p_jE_j\Bigr) 
-  h^q \Bigl(Y, \sum_{i=1}^r m_i D_i' \Bigr) \biggr| 
\leq C p (m+p)^{n-1}+ O(d^{n-1})
\end{equation}
for $d  \coloneqq m+p+1$ 
and
\begin{equation}
\label{eq easy 4'}
\biggl| \widehat{h}^q \Bigl(Y, \sum_{i=1}^{r} m_i D_i + \sum_{j=1}^s p_jE_j\Bigr) 
- \widehat{h}^q \Bigl(Y, \sum_{i=1}^r m_i D_i' \Bigr) \biggr| 
\leq  n! \, C p (m+p)^{n-1}.
\end{equation}

\end{lemma}

\begin{proof}
The bound \eqref{eq easy 4} follows directly from Proposition \ref{lemma perturbation3}
after choosing decompositions of $D_i$ and $E_j$ for all $i, j$. Then \eqref{eq easy 4'} is an asymptotic consequence of \eqref{eq easy 4}.
\end{proof}

\begin{prop} \label{lemma volume real divisors'}
For any $q \in \N$, the function $\widehat{h}^q$ is homogeneous of degree $n$ on $\Div(Y)_\R$ and continuous on every finite dimensional $\R$-subspace with respect to any norm. 
\end{prop}

\begin{proof}
To prove homogenity, we choose $\lambda >0$. For every non-zero $m \in \N$, there are $k_m \in \N$ and $r_m \in \R$ with $ m \lambda  = k_m +  r_m$ and $0 \leq r_m \leq 1$. By \eqref{eq easy 4}, we have 
\begin{equation} \label{homogenity 1}
|{h}^q(Y, m \lambda D) - h^q(Y,  k_m D)| 
\leq C r_m (k_m + r_m)^{n-1} +O(m^{n-1}) = O(m^{n-1}).
\end{equation}
Dividing \eqref{homogenity 1} by $m^n/n!= (k_m)^n/(n! \lambda^n) + O(m^{n-1})$ and passing to the $\limsup$, we get
\[
\widehat{h}^q(Y, \lambda D) \leq \lambda^n\widehat{h}^q(Y,  D).
\]
Replacing $D$ by $\lambda^{-1}D$, we get the reversed inequality for $\mu \coloneqq\lambda^{-1}$ instead of $\lambda$. This proves homogenity.  Continuity on finite dimensional subspaces follows from \eqref{eq easy 4'}.
\end{proof}

\begin{rem}  \label{compare with volumes}
If $Y$ is a projective variety over the field $k$, we call $\widehat{h}^0(Y,D)$ the {\it volume} of $D \in \Div(Y)_\R$ extending the classical notion from \ref{classical volume} to real Cartier divisors. Then we claim {that the $\limsup$ in the definition of $\vol$ is actually a limit, thus} 
\begin{equation} \label{limit for real volumes}
\vol(D)=\lim_{m \to \infty} \frac{h^0(mD)}{m^n/n!}. 
\end{equation} 
\end{rem}
\begin{proof}
For $D \in \Div(Y)$, this follows from a result of Cutkosky \cite[Thm.~8.1]{Cut15}. 
For $D \in \Div(Y)_\Q$, there is a non-zero $e \in \N$ with $eD$ represented by a Cartier divisor $D'$ on $Y$. Applying the previous case to $D'$ and using \eqref{eq easy 4}, we deduce that 
$$\vol(D')= \lim_{k \to \infty} \frac{h^0(kD')}{k^n/n!}=   \lim_{k \to \infty} \frac{h^0(kD'+ rD)}{k^n/n!}= e^n \lim_{k \to \infty} \frac{h^0((ke+ r)D)}{(ke+r)^n/n!}$$
for $r=0, \dots, e-1$. By homogenity of the volume, we get \eqref{limit for real volumes} for $D \in \Div(Y)_\Q$.  

To prove the claim for $D \in \Div(Y)_\R$, we choose a finite dimensional real subspace $W$ which has a basis $D_1, \dots, D_r$ in $\Div(Y)_\Q$ and with $D \in W $. For $\varepsilon >0$, pick $D' \in \Div(Y)_\Q$  with distance to $D$ in $W$ bounded by $\varepsilon$. By \eqref{eq easy 4}, there is $C \in \R_{\geq 0}$ independent of $\varepsilon$ and $m$ with $h^0(Y,mD)-h^0(mD')\leq C  \varepsilon m^n$ . Then  \eqref{limit for real volumes} for $D'$ yields  \eqref{limit for real volumes} for $D$.
\end{proof}

\subsection{Asymptotic formulas for families of real divisors}
\label{subsection asymptotic family}
In this subsection,  $Y$ is a projective variety over a field $k$. We will use the continuity of the asymptotic cohomological
functions in Proposition \ref{lemma volume real divisors'} to derive asymptotic
estimates for real divisors. Since we are using the asymptotic
cohomological functions we obtain only estimates up to
$o\left(m^n\right)$ and not up to $O\left(m^{n-1}\right)$, but these
will be enough for our applications.

\begin{prop}
\label{prop asymptotic}
For $D_1,\ldots,D_r \in \Div(Y)_\R$, there is $\rho\colon\N \to \R_{\geq 0}$ with $\rho(m) = o(m^n)$ for $m \to \infty$ such that for all non-zero $m_1,\ldots,m_r\in \N$ and 
$m\coloneqq \sum_{i = 1}^r m_i$, we have 
\begin{equation} \label{item asymptotic cohomology}
h^q\Bigl(Y,  {\sum_{i=1}^r} m_i D_i \Bigr) \leq \frac{m^n}{n!} \widehat{h}^q\Bigl(Y, { \sum_{i=1}^r} \frac{m_i}{m}D_i\Bigr) + \rho(m)
\end{equation}
and for  $q=0$, we even have 
$\big| h^0(Y,{ \sum_{i=1}^r} m_i D_i ) - \frac{1}{n!} \vol({\sum_{i=1}^r} m_i D_i ) \big| \leq \rho(m)$.
\end{prop}
\begin{proof}
Let us  prove the proposition by contradiction.
Then there are  $\alpha >0$ and some  sequences  
$(m_{i,k})_{k\in \N}$ in $\N \setminus \{0\}$ for $i=1 \ldots r$ such that  $m_k \coloneqq \sum_{i=1}^r m_{i,k} \to \infty$ 
 and 
\begin{equation}
\label{eq as vol 1}
 h^q\Bigl(Y, \sum_{i =1}^r m_{i,k} D_i\Bigr) - \frac{m_k^n}{n!}  
 \widehat{h}^q\Bigl(Y, \sum_{i =1}^r \frac{m_{i,k}}{m_k}D_i \Bigr)  \geq  
\alpha m_k^n.
\end{equation}
In case  $q=0$,  we replace the left side by its absolute value. Since for each $i,k$ we get 
$\frac{m_{i,k}}{m_k} \in [0,1]$, by compactness and up to considering subsequences, 
we may assume    
$\lim_{k \to \infty}\frac{m_{i,k}}{m_k} = c_i \in [0,1]$.
For $k\gg0$, the continuity of $\widehat{h}^q$ given in \eqref{eq easy 4'} yields 
\begin{equation}
\label{eq as vol 2}
h^q\Bigl(Y, \sum_{i =1}^r m_{i,k} D_i \Bigr) - \frac{m_k^n}{n!}\widehat{h}^q
\Bigl(Y,\sum_{i =1}^r c_i D_i\Bigr)  > \frac{\alpha}{2}m_k^n.
\end{equation}
In case $q=0$, this holds again with the absolute value of the left hand side. Using that $m_{i,k} = m_kc_i + (m_{i,k} - m_k c_i)$, Lemma \ref{lemma perturbation real} gives a
 $C \geq 0$ such that for all $k  \in \N$  
\[
\biggl| h^q\Bigl(Y,\sum_{i =1}^r m_{i,k} D_i\Bigr) -h^q\Bigl(Y,\sum_{i =1}^r m_kc_i D_i\Bigr)\biggr| \leq C \Bigl( \sum_{i =1}^r |m_{i,k} - m_k c_i|\Bigr)\cdot m_k^{n-1} +O(m_k^{n-1}).
\]
Since $\frac{m_{i,k}}{m_k} \xrightarrow[k]{} c_i$
it follows that $\sum_{i =1}^r | m_{i,k} - m_k c_i | = o(m_k)$ always for $k \to \infty$.  
Hence 
\begin{equation}
\label{eq as vol 3}
\biggl| h^q\Bigl(Y,\sum_{i =1}^r m_{i,k} D_i\Bigr) -h^q\Bigl(Y,\sum_{i =1}^r m_kc_i D_i\Bigr)\biggr|  = o( m_k^n)
\end{equation}
for $k \to \infty$. By definition of $\widehat{h}^q$ in \ref{defi rounds} and using $\sum_{i =1}^r m_kc_iD_i = m_k(\sum_{i =1}^r c_iD_i)$ we get 
\begin{equation}
\label{eq as vol 4}
h^q\Bigl(Y, \sum_{i =1}^r m_k c_iD_i\Bigr) - \frac{m_k^n}{n!} \widehat{h}^q\Bigl( Y,\sum_{i =1}^r c_i D_i \Bigr)  \leq o( m_k^n)
\end{equation}
for $k \to \infty$. In case $q=0$, the $\limsup$ in the definition of $\vol=\widehat{h}^0$ is a limit (see Remark \ref{compare with volumes}) and then \eqref{eq as vol 4} holds with the absolute value of the left side. 
Combining \eqref{eq as vol 3} with \eqref{eq as vol 4}, we get a contradiction to \eqref{eq as vol 2}. This proves the proposition.
\end{proof}

\subsection{Asymptotic formulas in the non reduced case}

\label{subsection nonreduced}
We fix the following notation for this subsection. The base is $S = \Spec(R)$ for a noetherian ring $R$   
and $Y$ is a projective scheme over $S$. 
We consider a coherent sheaf $\Fcal$ on $Y$ with support over a zero-dimensional subscheme of $S$. Let $n  \coloneqq \dim(\supp(\Fcal))$ and  
let $\{E_i\}_{i\in I}$ be the set of $n$-dimensional irreducible components of $\supp(\Fcal)$.
For each $i\in I$, let $\ell_i \coloneqq \length_{\Ocal_{Y,\xi_i}} (\Fcal_{\xi_i})$ where $\xi_i$ is the generic point of $E_i$. 

We also fix Cartier divisors $D_1, \dots, D_r$. For $i_1, \dots, i_n \in \{0, \dots, r\}$, we will use the intersection numbers 
\begin{equation}
\label{Mittwoch 3 August}
D_{i_1} \cdots D_{i_n}\cdot \Fcal = { \sum_{i\in I} } \ell_i D_{i_1} \cdots D_{i_n}\cdot E_i
\end{equation}
from \cite[\S VI.2]{Kol96}. We start with an asymptotic formula for the Euler characteristic $\chi$.

\begin{prop}
\label{prop asymptotic chi}
With the above notation, we have
\[ 
\chi\Bigl(Y, \Fcal\Bigl( { \sum_{i=1}^r} m_i D_i\Bigr)\Bigr) = \frac{1}{n!} \Bigl( {\sum_{i=1}^r} {m_i}D_i \Bigr)^n \cdot \Fcal + O( m^{n-1}).
\]
\end{prop}
\begin{proof}
This follows from \cite[Thm.~VI.2.13]{Kol96} using the definition of intersection numbers in \cite[VI.2.6]{Kol96}.
\end{proof}

\begin{prop}
\label{cor cohomology}
For $q \in \N$, there is $\rho\colon \N \to \R_{\geq 0}$ with $\rho(m)=o(m^n)$ such that for all $m_1,\ldots,m_r \in \N \setminus \{0\}$ and $m  \coloneqq \sum_{j=1}^r m_j$, we have 
\begin{align*} 
h^q\Bigl(Y,\Fcal\Bigl({\sum_{j=1}^r} m_j D_j\Bigr)\Bigr) \leq 
\frac{1}{n!} \sum_{i\in I} \ell_i \widehat{h}^q \Bigl(E_i, \Ocal\Bigl( { \sum_{j=1}^r}  {m_j} D_j\Bigr) \Big|_{E_i} \Bigr) + \rho(m).
\end{align*}
\end{prop}

\begin{proof}
By assumption, $E_i$ is lying over a closed point $x_i$ of $S$ and hence we may view $E_i$ as a projective variety over the residue field of $x_i$. 
The result now follows from Lemma \ref{lemma bound cohomology non reduced} and Proposition \ref{prop asymptotic}.
\end{proof}

\begin{cor}
\label{cor cohomology nef}\label{volisenergy9}
If $D_1, \dots , D_r$ are nef and $q \geq 1$, then there are 
functions $\rho_i\colon \N \to \R_{\geq 0}$ 
$(i=1,2)$ with $\rho_i(m)=o(m^n)$ such that for all $m_1,\ldots,m_r \in \N \setminus \{0\}$ and $m  \coloneqq \sum_{j=1}^r m_j$, we have
\[ 
h^q\Bigl(Y, \Fcal\Bigl({\sum_{j=1}^r} m_j D_j \Bigr)\Bigr) = \rho_1(m)
\]
and
\[ 
h^0\Bigl(Y, \Fcal\Bigl({\sum_{j=1}^r} m_j D_j \Bigr)\Bigr) = \frac{1}{n!} 
\Bigl( {\sum_{j=1}^r} m_j  D_j \Bigr)^n  \cdot \Fcal
+\rho_2(m).
\]
\end{cor}
\begin{proof}
Again, we may view any $E_i$ as a projective variety over a suitable field. 
Note that the asymptotic Riemann--Roch formula in \cite[Thm.~VI.2.15]{Kol96} yields $\widehat{h}^q(E_i,D)=0$ for any nef divisor $D$ on $E_i$ and hence the first claim follows from
 Proposition \ref{cor cohomology}. The second claim follows from the first claim and Proposition \ref{prop asymptotic chi}.
\end{proof}

\section{{Non-archimedean} volumes and energy}
\label{section non-archimedean volume}

In this section, $K$ is a discretely valued complete field with $-\log(|\pi|)=1$ for a uniformizer $\pi$.
We consider  a projective 
variety $X$ over $K$ of dimension $n$ with a line bundle $L$. 
All metrics on line bundles are assumed to be continuous. 
The length of a $\Ko$-module $M$ is denoted by $\length(M)$. 
We will use the algebraic volume $\vol(L)$ 
from \ref{classical volume}. 

\subsection{{Non-archimedean} volumes}
\label{subsection non-archimedean volume}

\begin{defi}
\label{defi length}
If $V$ is a finite-dimensional $K$-vector space,
a \emph{lattice of $V$} is a free $\Ko$-submodule {of} $\Lambda \subset V$ 
with $K$-span $V$.   
If $\Lambda_2 \subset \Lambda_1 \subset V$ are  lattices of $V$, then 
$\length( \Lambda_1 / \Lambda_2)$ is finite since $\Lambda_1 / \Lambda_2$  is a finitely generated torsion $\Ko$-module. 
If $\Lambda_1,\Lambda_2$ are any lattices of $V$, we choose a lattice $\Lambda_3$  contained in both $\Lambda_1$ and $\Lambda_2$ and we set 
\[ \length( \Lambda_1 / \Lambda_2) = \length ( \Lambda_1 / \Lambda_3) - \length (\Lambda_2 / \Lambda_3)\in\Z.\]
\end{defi}

This is independent of  the choice of $\Lambda_3$.
Observe that $\length( \Lambda_1 / \Lambda_2)$ might become
negative.
{
Recall from \ref{continuous metrics} that $\Hhat(X,L,\metr)  \coloneqq  \{ s\in H^0(X,L) \st 
\|s\|_{\rm sup}\leq 1\}$ is a lattice of $H^0(X,L)$.
}
\begin{defi}
\label{defi non-archimedean volume}
If $\| \ \|_1$ and $\| \ \|_2$ are two metrics on $\Lan$, we define the \emph{non-archimedean volume of $L$ with respect to $\| \ \|_1$ and $\| \ \|_2$} by 

\begin{align*}
\vol(L,\| \ \|_1, \| \ \|_2) = \limsup_{m \to \infty} \frac {n!} {m^{n+1}} \cdot \length \left( \frac {\Hhat(X,L^{\otimes m},\metr_1^{\otimes m})} { \Hhat(X,L^{\otimes m},\metr_2^{\otimes m})} \right).
\end{align*}
\end{defi}

Often, we will write $\vol(\| \ \|_1, \| \ \|_2)$  instead of $\vol(L,\| \ \|_1, \| \ \|_2)$. 
For the following result, recall that we have 
$|\pi|^{-1}=\exp(1)$ by our normalization of the valuation
on $K$.

\begin{lemma}
\label{volume of proportional metrics}
For $t \in \R$, we have 
$$\vol(L, e^{-t}\metr_1, \metr_2) = \vol(L, \metr_1, e^{t}\metr_2)=t \vol(L) + \vol(L, \metr_1,\metr_2) .$$
\end{lemma}

\begin{proof}
Note that $M_m  \coloneqq \Hhat(X,L^{\otimes m},\metr_1^{\otimes m})$ and 
$M_m'  \coloneqq \Hhat(X,L^{\otimes m},\metr_2^{\otimes m})$ are free $\Ko$-modules of the same rank $h^0(X,L^{\otimes m})$. 
We first assume that $t= k \in \Z$. 

Then the additivity of the length and 
$\Hhat(X,L^{\otimes m},e^{-km}\metr_1^{\otimes m}) = \pi^{-km} M_m$
show 
\begin{equation} \label{length of proportional lattices}
\length \left(\Hhat(X,L^{\otimes m},e^{-km}\metr_1^{\otimes m}) / M_m' \right)= km\, h^0(X,L^{\otimes m})+\length (M_m/M_m').
\end{equation}
By \ref{classical volume}, we have
$$\vol(L)= \lim_{m \to \infty} \frac{h^0(X,L^{\otimes m})}{m^n/n!}$$
and $\vol(L, e^{-k}\metr_1, \metr_2)= k \vol(L) + \vol(L, \metr_1,\metr_2)$ follows from \eqref{length of proportional lattices} and the definition of the {non-archimedean} volumes. Similarly, we prove the other equality. 

If $t \not \in \Z$, then $\pi^{-\lfloor tm \rfloor } M_m \subset \Hhat(X,L^{\otimes m},e^{-tm}\metr_1^{\otimes m}) \subset  \pi^{-\lceil tm \rceil} M_m$ and the claim follows from a sandwich argument similarly as above.  
\end{proof}

\begin{prop}  \label{lemma continuity non-archimedean volume}
For metrics $\metr_1, \metr_2$ on $\Lan$, we have the following properties:
\begin{itemize}
\item[(a)] $\vol(\| \ \|_1, \| \ \|_2)$ is  monotone decreasing in $\metr_1$ and a monotone  increasing  in $\metr_2$.
\item[(b)]  $\vol(\| \ \|_1, \| \ \|_2)$ is finite and  continuous in $(\metr_1,\metr_2)$. 
\end{itemize}
\end{prop}

\begin{proof}
Property (a) is obvious. Finiteness in (b) and the inequality 
\begin{equation} \label{Lipschitz}
\left|  \vol(\metr_1',\metr_2) - \vol(\metr_1, \metr_2) \right| \leq \vol(L)d(\metr_1,\metr_1')
\end{equation}
for any metric $\metr_1'$ on $\Lan$ follow from an easy sandwich argument based on (a) and Lemma \ref{volume of proportional metrics}, where $d$ is the distance from \ref{distance of metrics}. 
Similarly as in \eqref{Lipschitz}, $|  \vol(\metr_1',\metr_2) - \vol(\metr_1, \metr_2)|$ is bounded by $\vol(L)d(\metr_2,\metr_2')$ and hence continuity in (b) follows. 
\end{proof}

 \begin{lemma}
\label{lemma twisted volume}
Let $L$ and $M$ be line bundles on $X$. Then we have 
\[
\limsup_{m \to \infty} \left| 
\frac{n!} {m^n} \cdot \length 
\left( \frac {\Hhat(X,M \otimes L^{\otimes m}, \metr_1 \otimes \metr^{\otimes m}) } { \Hhat(X,M \otimes L^{\otimes m}, \metr_2 \otimes \metr^{\otimes m})} \right)
\right| \leq \vol(L) d(\metr_1,\metr_2)
\]
for any metrics $\metr$ on $\Lan$ and $\metr_1$, $\metr_2$  on $M^{\rm an}$.
\end{lemma}

\begin{proof} 
This is a twisted variant of \eqref{Lipschitz} which follows along the same lines. 
\end{proof}

\begin{rem} \label{volume and big}
Let $L$ be a line bundle on $X$ which is not big. By definition, this means that $\vol(L)=0$. It follows easily from Lemma \ref{volume of proportional metrics}, Proposition \ref{lemma continuity non-archimedean volume} and a sandwich argument that $\vol(L,\metr_1,\metr_2)=0$ for all continuous metrics $\metr_1,\metr_2$ on $\Lan$. 
\end{rem}

\begin{rem} \label{number field case}
Let us describe how the \emph{non-archimedean
volume} is related to the \emph{$\chi$-arithmetic volume}
which is studied in Arakelov theory.
The precise relation is given in formula
\eqref{generalization of comp} below.
We assume in this remark that $F$ is a number field with ring of integers $\Ocal_F$ and with set of places $M_F$.

Let $L$ be a line bundle on an $n$-dimensional projective variety 
$X$ over $F$ endowed with an adelic metric which means that we have 
a continuous metric $\metr_w$ on $L \otimes_F F_w$
for the completion $F_w$ of any $w \in M_F$ and we assume that 
there is a finite set $S$ of $\Spec(\Ocal_F)$ such that the metric 
$\metr_w$ is induced by a single model of $(X,L)$ over 
$\Spec(\Ocal_F) \setminus S$ for all non-archimedean places 
$w \not\in S$. 
We denote the resulting metrized line bundle by $\overline{L}$ 
and we set $E:=H^0(X,L)$. For $w \in M_F$, let $B_w$ be the 
unit ball in 
$E \otimes_F F_w = H^0(X \otimes_F F_w, L \otimes_F F_w)$ with 
respect to the sup-norm. 
Observe that $B_w$ is a finitely generated $F_w^\circ$-module. 
We note that $\Lambda := \bigcap_{\text{$w$ finite}} B_w \cap E$ 
is a lattice in 
$E \otimes_\Q \R = \prod_{w | \infty} H^0(X \otimes_F F_w, 
L \otimes_F F_w)$ \cite[Proposition C.2.6]{BomGub} and we set
\[
\chi(X, \overline{L}) :=  \log \left( \frac{\vol\bigl(
\prod_{w | \infty} B_w\bigr)}{\covol(\Lambda)} \right)
\]
where the volume and the covolume are computed with respect to the same Haar measure on $E \otimes_\Q \R$. 
If the adelic metric is induced by a normal $\Ocal_F$-model $(\Xcal,\Lcal)$ of $(X,L)$, then $\Lambda = H^0(\Xcal,\Lcal)$ (see Lemma \ref{sections and metrics} and Remark \ref{non-complete generalization}).

Now we assume that $L$ is ample. Then we have the {\it $\chi$-arithmetic volume}
\[
\widehat{\rm vol}_{\chi}(X,\overline{L}):= \limsup_{m \to
\infty} \frac{(n+1)!}{m^{n+1}}  \chi(X,\overline{L}^{\otimes m})
\]
considered in Arakelov theory. 
It agrees with the \emph{logarithm of the sectional capacity} 
studied in the book of Rumely, Lau and Varley \cite{RVL}.
It follows from \cite[Thm.~B]{RVL} that the limsup in the 
definition is actually a limit. Zhang's extension 
\cite[Thm.~1.4]{zhang-jams} of the arithmetic Hilbert--Samuel formula of Gillet--Soul\'e shows that 
$\widehat{\rm vol}_{\chi}(X,\overline{L})$ 
is finite in case of a semipositive adelic metric and hence the 
continuity argument in 
\cite[Sect.~5]{CT} shows that 
$\widehat{\rm vol}_{\chi}(X,\overline{L}) \in \R$ is finite for any adelic metric on the ample line bundle $L$.

Let us now fix a non-archimedean place $v$ of $F$.
We consider two continuous metrics $\metr_v$ and $\metr_v'$ on 
$L \otimes_F F_v$ at the fixed non-archimedean place $v$ inducing unit balls $B_v$ and $B_v'$ in $H^0(X \otimes_F F_v, L \otimes_F F_v)$ with respect to the sup-norms. 
We extend the metrics to adelically metrized line bundles $\overline{L}$ and $\overline{L}'$ using the same metrics $\metr_w$ for all places $w \neq v$. 
From Arakelov theory on the arithmetic curve $\Spec(\Ocal_F)$, we get the formula 
\begin{equation} \label{chi and length}
\chi(X,\overline{L}) - \chi(X,\overline{L}') = \log(\# \widetilde{F_v}) \cdot \ell_{F_v^\circ}(B_v/B_v').
\end{equation}
which holds without assuming $L$ ample and which can be deduced from 
the Riemann--Roch formula given in \cite{Gaudron} before Lemma 4.2. If $L$ is ample, then we apply  \eqref{chi and length} for $\overline{L}^{\otimes m}$ and $\overline{L}'^{\otimes m}$, multiply it with $\frac{(n+1)!}{m^{n+1}}$ and pass to the limit. This proves 
the formula
\begin{equation} \label{generalization of comp}
\widehat{\vol}_{\chi}(X,\overline{L}) - \widehat{\vol}_\chi(X,\overline{L}')   =
 (n+1) \log(\# \widetilde{F_v}) \cdot \vol(L, \metr_v, \metr_v')
\end{equation}
which describes the non-archimedean volume in the
number field case as a \emph{localized} $\chi$-arithmetic volume.
\end{rem}
 
\begin{rem} \label{conjecture}
We conjecture that the limsup in the definition of $\vol(L, \metr_1,\metr_2)$ is always a limit. In the case of a non-archimedean completion $K$ of a number field $F$, with $X$ and 
$L$ defined over $F$ {and with $L$ ample}, this follows from the argument deducing \eqref{generalization of comp} from \eqref{chi and length}. In Theorem \ref{cor3 energy} and in Corollary \ref{cor 4 energy}, we will prove special cases of the conjecture.

{A referee pointed out that a result of Chen and 
Maclean  \cite[Corollary 4.6]{chen-maclean2015} proves 
this conjecture if the projective variety $X$ contains 
a $K$-rational regular point. Indeed, the arguments in 
Remark \ref{volume and big} show that we may assume $L$ big. 
By \cite[footnote 10 on p.~388]{chen-maclean2015}, 
the conditions (a)--(c) in \cite[p.~385]{chen-maclean2015} 
are satisfied for the complete graded linear system induced by 
$(H^0(X,L^{\otimes m}))_{m \in \N}$ and hence we may apply 
\cite[Corollary 4.6]{chen-maclean2015} to get the existence of the 
limit in the definition of $\vol(L, \metr_1,\metr_2)$. 
For this last step, one has to ensure that the considered 
sequence in \cite[Corollary 4.6]{chen-maclean2015} is 
asymptotically equal to  
\[
\left(\frac {1}{mh^0(X,L^{\otimes m})} \cdot \length \biggl( \frac {\Hhat(X,L^{\otimes m},\metr_1^{\otimes m})} { \Hhat(X,L^{\otimes m},\metr_2^{\otimes m})}\biggl) \right)_{{m\in \N \setminus \{0\}}}
\]
which follows from \cite[Proposition 2.21]{boucksom-eriksson2018}. 
To apply the latter, we note that the determinant norms in \cite[\S 3]{chen-maclean2015} and in \cite[\S 2]{boucksom-eriksson2018} agree as all ultrametric norms on a finite dimensional $K$-vector space are diagonalizable \cite[Example 1.12]{boucksom-eriksson2018} 
and we have the same concrete formula for diagonalizable norms. 
We thank the referee for hinting us to the  
{reference} \cite[Corollary 4.6]{chen-maclean2015}.}

{Note also that Boucksom and Eriksson \cite[Lemma 8.8]{boucksom-eriksson2018} prove the conjecture in case of an ample line bundle on a smooth projective variety $X$. In fact, they show it not only for continuous metrics, but also for bounded metrics. Moreover, the existence of the limit in \cite[Corollary 4.6]{chen-maclean2015} and in \cite[Lemma 8.8]{boucksom-eriksson2018} holds over any 
{not necessarily discretely valued complete} 
non-archimedean field $K$.} 
\end{rem}

\subsection{Volumes and semipositive metrics}
\label{subsection volume semipositive}

{In this subsection, we consider a normal projective variety $X$ over the complete discretely valued field $K$.} 

{If $M$ is a $\Ko$-module and $a\in \Ko$ we set 
\[M_{a-\tor} = \{ m\in M \st am=0\}.\]}

\begin{lemma}
\label{lemma basic torsion}
Let $M$ be a $\Ko$-module of finite type.
For any $\alpha \in \N$, we have 
\[ \length (M_{\pi^\alpha-\tor}) \leq \length (M/\pi^\alpha M).\]
\end{lemma}
\begin{proof}
This follows from the classification of modules of finite type over a PID. 
\end{proof}

Recall from \eqref{new-definition-energy}  that we have defined the energy $E(L, \metr_1,\metr_2)$ of {continuous semipositive} metrics $\metr_1,\metr_2$ on a line bundle $L$ over $X$. 
The following proposition is our key point to interpret the energy as a non-archimedean volume.

\begin{prop}
\label{prop energy 2'}
Let $L$ be {a} line bundle on $X$ and let $\KX$ be a normal model of $X$.
We consider  
{nef} models $\KL_1$ and $\KL_2$ of $L$ and 
we write $\KL_1 \otimes \KL_2^{-1} = \Ocal(D)$ for 
some  vertical 
Cartier divisor $D$ on $\KX$.
In addition, let $\Mcal$ be a line bundle on $\KX$ with generic fibre $M \coloneqq \Mcal_{|X}$. Then we have 
\begin{align*}
 E(L, \metr_{\KL_1}, \metr_{\KL_2}) =  \lim_{m\to 0} \frac{n!}{m^{n+1}}  \length 
\left( \frac {\Hhat(X,M \otimes L^{\otimes m},\metr_\Mcal \otimes \metr_{\KL_1}^{\otimes m})} {\Hhat(X,M \otimes L^{\otimes m},\metr_\Mcal \otimes \metr_{\KL_2}^{\otimes m})} \right).
\end{align*}
\end{prop}

\begin{proof}
First, we reduce the claim to the case when $D$ is an 
effective vertical Cartier divisor.
There is a $k \in \N$ such that $D' \coloneqq\div(\pi^k)+ D$ is an effective Cartier divisor {and for $\KL_1'   \coloneqq \KL_1(\div(\pi^k)) \simeq \KL_1$ we get 
$\Ocal(D')= \KL_1' \otimes \KL_2^{-1}$.
Note that $\KL_1'$ is still 
{nef} and $\metr_{\KL_1'}=  |\pi|^{k} \metr_{\KL_1}$.}
{Using the definition of the energy and \ref{chambert loir measure for metrized line bundles}(i), we get 
\[
E(L, \metr_{\KL'_1}, \metr_{\KL_2})  = 
k L^n + E(L, \metr_{\KL_1}, \metr_{\KL_2}).
\]
The same argument as for \eqref{length of proportional lattices} and then \ref{classical volume} and Remark \ref{remark volume nef} yield  
\[
 \length\left( \frac {\Hhat(X,M \otimes L^{\otimes m}, \metr_\Mcal \otimes \metr_{\KL_1'}^{\otimes m})} {\Hhat(X,M \otimes L^{\otimes m},\metr_\Mcal \otimes \metr_{\KL_1}^{\otimes m})} \right)
 = km\, h^0(X,M \otimes L^{\otimes m}) \underset{m \to +\infty}{\sim} k \frac{m^{n+1}}{n!}L^n.
\]
Hence the claim for $D'$ implies the claim for $D$, and we can replace $D$ by $D'$.}

So we may assume that $D$ is an effective vertical Cartier divisor. 
Let $s_D\in \Gamma(\KX,\Ocal(D))$ denote the canonical 
global section of $\Ocal(D)$. 
Note that $\mbox{div}\,(s_D)=D$. 
Let $\varphi_D$ denote the model function associated with $D$.
For $j \in \{0, \dots, m\}$, we use the notation
\begin{equation} \label{definition of twisted sheafs} \Fcal_j^{(m)}  \coloneqq \Mcal \otimes {\KL_1^{\otimes j}} \otimes {\KL_2^{\otimes {m-j}}}.
\end{equation}
For $j \in \{1, \dots, m\}$, we  
consider the short exact sequence 
\begin{equation}
\label{eq ses energy}
0 \to \Fcal^{(m)}_{j-1}  
\xrightarrow[ ]{\otimes s_D} \Fcal_j^{(m)}
 \longrightarrow
 \Fcal_j^{(m)}|_{D}  \longrightarrow 0.
\end{equation}
{The associated long exact sequence in cohomology gives 
\begin{multline}\label{volisenergyeq1} 
0 \longrightarrow \Gamma(\KX, \Fcal^{(m)}_{j-1} ) 
\xrightarrow[ ]{\otimes s_D} 
\Gamma(\KX, \Fcal_j^{(m)})\longrightarrow 
\Gamma(D,\Fcal_j^{(m)}) 
\longrightarrow H^1( \KX, \Fcal^{(m)}_{j-1})  
\longrightarrow \cdots 
\end{multline}
Let us pick $\alpha \in \N$ such that $0\leq \varphi_D \leq \alpha$. 
Using that $\KX$ is normal, Proposition \ref{effective model fct} yields $\pi^\alpha \in\Jcal_D$, {where $\Jcal_D$ is the ideal sheaf of the closed subscheme $D$}.
Hence $D$ is in a natural way a scheme of finite type over 
$S \coloneqq \Spec ( \Ko / \pi^\alpha \Ko)$.
The $K^\circ$-module $\Gamma(D,\Fcal_j^{(m)})$ is 
$\pi^\alpha$-torsion as $\pi^\alpha \in\Jcal_D$. 
Since the restrictions of $\KL_1$ and $\KL_2$ to 
$\KX_S  \coloneqq \KX\times_\Ko {S}$ are nef, 
Corollary \ref{cor cohomology nef} 
 yields that 
\begin{equation}\label{volisenergyeq2}
\length \bigl(H^1 ( \KX_S, \Fcal^{(m)}_{j-1})\bigr)   =o(m^n).
\end{equation}
From the short exact sequence
\[
0\longrightarrow
\Fcal^{(m)}_{j-1} \stackrel{\cdot\pi^\alpha}{\longrightarrow}
\Fcal^{(m)}_{j-1}    \longrightarrow
\Fcal^{(m)}_{j-1}|_{\KX_S}\longrightarrow 0
\]
we get the  exact sequence 
\[
H^1 ( \KX, \Fcal^{(m)}_{j-1} )
\stackrel{\cdot\pi^\alpha}{\longrightarrow}
H^1 ( \KX, \Fcal^{(m)}_{j-1} )\longrightarrow
H^1 ( \KX_S,  \Fcal^{(m)}_{j-1})
\]
and hence the induced homomorphism 
\begin{equation*}
\label{eq trick torsion}
 H^1 \left( \KX, \Fcal^{(m)}_{j-1} \right) / \pi^\alpha H^1 \left( \KX, \Fcal^{(m)}_{j-1} \right)
 \hookrightarrow   H^1 \left( \KX_S, \Fcal^{(m)}_{j-1} \right)
\end{equation*}
is injective.
Together with Lemma \ref{lemma basic torsion} and \eqref{volisenergyeq2} this shows that
\begin{equation}\label{volisenergyeq3}
\length \bigl(H^1 ( \KX, \Fcal^{(m)}_{j-1} )_{\pi^\alpha-\tors}   
\bigr)=o(m^n).
\end{equation} 
Then \eqref{volisenergyeq1} and \eqref{volisenergyeq3} show that 
\begin{equation} \label{crucial quotient estimate}
\length \left( \Gamma(\KX,\Fcal^{(m)}_j) / \Gamma(\KX,\Fcal^{(m)}_{j-1}) \right) = \length \left( \Gamma(D,\Fcal^{(m)}_j) \right) + o(m^n).
\end{equation}}

Let $D_1$ be a Cartier divisor with ${\KL_1}= \Ocal(D_1)$ and  $D_2  \coloneqq D_1-D$. 
Observing \eqref{Mittwoch 3 August}, Corollary \ref{volisenergy9} gives
\begin{equation}\label{volisenergyeq4}
\length \left(\Gamma(D,  \Fcal^{(m)}_j ) \right)
= \frac{m^n}{n!} 
\biggl(\frac{j}{m}D_1 + \Bigl(1-\frac{j}{m}\Bigr)D_2\biggr)^n \cdot D +o(m^n).
\end{equation}

It follows from Lemma  \ref{lemma volume models} that  
$\Hhat(X,M \otimes L^{\otimes m},\metr_\Mcal \otimes \metr_{\KL_1}^{\otimes m}) = \Gamma(\KX,\Fcal^{(m)}_m)$ and 
$\Hhat(X,M \otimes L^{\otimes m},\metr_\Mcal \otimes \metr_{\KL_2}^{{\otimes m}}) = \Gamma(\KX,\Fcal^{(m)}_0)$. 
Hence we have to show that 
\begin{equation} \label{energy limit to show}
\frac{1}{n!}
E(L, \metr_{\KL_1}, \metr_{\KL_2}) =
\lim_{m\to\infty}\frac{1}{m^{n+1}}\length \Bigl( \Gamma(\KX,\Fcal^{(m)}_m) / \Gamma(\KX,\Fcal^{(m)}_0)
 \Bigr).
\end{equation}
Additivity of length, \eqref{crucial quotient estimate} 
and \eqref{volisenergyeq4}
yield  
\[
\frac{1}{m^{n+1}}\length \Bigl( \Gamma(\KX,\Fcal^{(m)}_m) / \Gamma(\KX,\Fcal^{(m)}_0)
 \Bigr)=
\frac{1}{n!\, m} \sum_{j=1}^m  \biggl(\frac{j}{m}D_1 + \Bigl(1-\frac{j}{m}\Bigr)D_2\biggr)^n \cdot D          +o(1).
\]
The limit for $m\to\infty$ exists and is given by the 
sum of Riemann integrals
$$\frac{1}{n!}  \int_{0}^1\
\Bigl(tD_1 + (1-t)D_2\Bigr)^n \cdot D \, dt 
= \frac{1}{n!} \sum_{k=0}^n \binom{n}{k} \, \int_0^1  t^k (1-t)^{n-k}dt \, D_1^k \cdot D_2^{n-k} \cdot D.
$$ 
Using the identity  $\int_0^1 (1-t)^k t^{n-k} dt = ((n+1) \binom{n}{k})^{-1}$, we get 
$$\lim_{m\to\infty}\frac{1}{m^{n+1}}\length \Bigl( \Gamma(\KX,\Fcal^{(m)}_m) / \Gamma(\KX,\Fcal^{(m)}_0)
 \Bigr) ={\frac{1}{n!}}\frac{1}{n+1} \sum_{k=0}^n   D_1^k \cdot D_2^{n-k} \cdot D$$
and hence \eqref{energy limit to show} follows from \eqref{-algebraic-definition-energy}.  
\end{proof}

\begin{theo}\label{cor3 energy}
Let $L$ be {a} line bundle on {the normal} projective variety $X$ and let $\| \ \|_1$ and $\| \ \|_2$ be continuous semipositive metrics on $L^{\an}$. 
Then we have
\begin{equation}\label{vol-is-energy-main-thm}
\vol(L,\metr_1,\metr_2) = E(L, \metr_1,\metr_2).
\end{equation}
Furthermore under our assumptions the $\limsup$ in the
definition of the non-archimedean volume is  a limit.
\end{theo}

\begin{proof}
{In the following, let $\varphi  \coloneqq -\log(\frac{\| \ \|_1}{\| \ \|_2})$. 
We first prove the claim 
for semipositive model metrics.
Then there exist an integer $k\in \N$, 
nef models $\Ncal_1$ and $\Ncal_2$ of the line bundle {$N  \coloneqq L^{\otimes k}$} such that
${\| \ \|_1^{\otimes k}} = \| \ \|_{\Ncal_1}$ and 
${\| \ \|_2^{\otimes k}} = \| \ \|_{\Ncal_2}$. 
We fix some $r \in \{0,\ldots,k-1\}$ which will play the role of the remainder in the euclidean division by $k$. 
Moreover we fix a model $\Mcal$ of $L^{\otimes r}$.
To have all our models of line bundles defined on the same normal model $\KX$, we pass to a common finer model. 
There is now a vertical Cartier divisor $D$ on $\KX$ such that $\Ocal(D)=\Ncal_1 \otimes \Ncal_2^{-1}$.
Note that we have $\varphi_D = k\varphi$.}

 Then it is enough to study the arithmetic progression made of the integers $m$ of the form 
$m= kq +r$ for $q\in \N$.
By Lemma \ref{lemma twisted volume}, we note that both 
\begin{equation*} 
\label{minor twist 1}
\length \left( \frac {\Hhat(X,L^{\otimes m}, \metr_1^{\otimes m})} {\Hhat(X,L^{\otimes r} \otimes L^{\otimes kq}, \metr_\Mcal  \otimes \metr_{1}^{{\otimes kq}})} \right) 
\text{ and }  
\length \left( \frac {\Hhat(X,L^{\otimes r} \otimes L^{\otimes kq}, \metr_\Mcal \otimes \metr_{2}^{{\otimes kq}})} {\Hhat(X,L^{\otimes m}, \metr_2^{\otimes m})} \right) 
\end{equation*}
equal $O(q^n)$.
Together with additivity of length and ${\metr_i^{\otimes k}}=\metr_{\Ncal_i}$, we get

\begin{align*}
\length \left( \frac {\Hhat(X,L^{\otimes m}, \metr_1^{\otimes m})} {\Hhat(X,L^{\otimes m}, \metr_2^{\otimes m})} \right) 
= \length \left( \frac {\Hhat(X,L^{\otimes r} \otimes L^{\otimes kq}, \| \ \|_\Mcal \otimes \| \ \|_{\Ncal_1}^{\otimes q})} 
{\Hhat(X,L^{\otimes r} \otimes L^{\otimes kq},\metr_\Mcal \otimes \metr_{\Ncal_2}^{\otimes q})} \right) + O(q^n). 
\end{align*}
By Proposition \ref{prop energy 2'}, $\varphi_D=k\varphi$ and the {homogeneity} of the energy, we deduce 
\begin{multline*}
 \length \left( \frac {\Hhat(X,L^{\otimes m}, \metr_1^{\otimes m})} {\Hhat(X,L^{\otimes m}, \metr_2^{\otimes m})} \right)  
=  \frac{q^{n+1}}{n!}   E(L^{\otimes k}, \metr_{\Ncal_1}, \metr_{\Ncal_2}) + o((kq)^{n+1} ) \\
=  \frac{q^{n+1}k^{n+1}}{n!}  E(L, \metr_1, \metr_2) +o((kq)^{n+1} )
 = \frac{ m^{n+1}}{n!} E(L, \metr_1, \metr_2) +o(m^{n+1} )
 \end{multline*}
along the arithmetic progression $(m=kq +r)_{q\in \N}$. 
{This proves the claim for model metrics.}

{Arbitrary continuous semipositive metrics  on $\Lan$ are uniform limits of semipositive model metrics on $\Lan$.}
Then the formula in the theorem follows from the first case as both the non-archimedean volume and 
the Chambert-Loir measure are continuous in $( \metr_1, \metr_2)$ (see Proposition \ref{lemma continuity non-archimedean volume} and 
\ref{chambert loir measure for metrized line bundles}).

It remains to see that the $\limsup$ in the
definition of the non-archimedean volume is  a limit. 
We choose a rational number $\varepsilon>0$.
{For $i=1,2$, there is a semipositive model metric $\metr_i'$ on $\Lan$ with distance to $\metr_i$  bounded by $\varepsilon$ and hence 
$e^{-\varepsilon} \metr_i' \leq \metr_i \leq e^{\varepsilon} \metr_i'$. 
As $e^{\pm\varepsilon} \metr_i'$ are semipositive model metrics, we deduce easily from a sandwich argument, from the first case and using $\varepsilon \to 0$ that the $\limsup$
is  a limit.}
\end{proof}

\begin{rem} 
As S\'ebastien Boucksom pointed out to us, in the proof of \cite[Lemma 3.5]{DEL00}, one can find arguments involving remainders in Euclidean divisions which are similar to some arguments in the proof of Theorem \ref{cor3 energy}.

The kind of use of Riemann sums made in the end of the proof of 
{Proposition \ref{prop energy 2'}}
already appeared in the literature on algebraic volumes.
See for instance \cite[Example 2.3.6]{Laz1} and \cite[Example 2.2]{ELMN05}.

{There is a description of the non-archimedean volume in terms of the energy for arbitrary continuous metrics} if the residue characteristic of $K$ is zero and if $X$ is a smooth projective variety.
Moreover,  the $\limsup$ in the
definition of the non-archimedean volume is again a limit. These results will be shown in Corollary \ref{cor 4 energy}.
\end{rem}

\section{Differentiability}
\label{section differentiability}

As usual, $K$ is a complete discretely valued field with valuation ring $\Ko$. 
Recall that we normalized our absolute value such that $-\log \vert \pi \vert =1$ 
for a uniformizer $\pi$.
Let $X$ be a projective variety over $K$ of dimension $n$.  
In this section, we consider  projective  $\Ko$-models $\KX$ of
$X$. The special fibre will be denoted by $\KX_s$. This is a scheme of
finite type over the residue field $\Kt$, but not necessarily reduced. We denote the
irreducible components of $\KX_s$ by $(E_i)_{i \in I}$ and let $b_i$
denote the multiplicity of $\KX_s$ in $E_i$.

\subsection{Upper-bound{s} for the first cohomology group}
\label{subsection first}
In the following, we will use the notations introduced
in \ref{Neron Severi}.
{Given Cartier divisors $D_1,\ldots,D_n$ on a model $\KX$ of $X$ 
we denote  by $\{D_1\}\cdots \{D_n\}$
the algebraic intersection number in the generic fibre.}

\begin{lemma}
\label{lemma first}
Let $D,M_1, M_2$ be nef divisors and let $\Ncal$ be any line bundle on $\KX$. 
{There exists a function $\rho\colon \N\to \R$ with $\rho(m)=o(m^n)$ as
$m\to \infty$ such that
\[ 
\dim_{\Kt} \Bigl( H^1\bigl(\KX, \Ncal(mD +j(M_1-M_2))\bigr)\otimes_\Ko \Kt \Bigr) \leq 
\frac{m^n}{n!} n \{ D+M_1\}^{n-1}\cdot \{M_2\} + \rho(m).
\]
holds for all $m\in\N$ and all $j\in\{0,\ldots, m\}$.}
\end{lemma}

\begin{proof}
{We will use the notation $\Fcal_{j,m} \coloneqq \Ncal(mD +j(M_1-M_2))$.}  
Let $\pi$ be a uniformizer of the discrete valuation ring $\Ko$ and let $M \coloneqq M_1- M_2$.   
The short exact sequence 
\[ 
0 \longrightarrow \Fcal_{j,m} \stackrel{\cdot \pi}{\longrightarrow} \Fcal_{j,m} \longrightarrow \Fcal_{j,m}|_{\Xs} \longrightarrow 0
\]
yields the long exact sequence 
\[ 
\ldots \longrightarrow H^1(\KX , \Fcal_{j,m}) \stackrel{\cdot \pi}{\longrightarrow} 
H^1(\KX , \Fcal_{j,m}) \longrightarrow 
 H^1(\Xs, \Fcal_{j,m}|_{\Xs}) \longrightarrow  \ldots 
\]
Forming the cokernel of the first map, we obtain an injection 
\[
 H^1(\KX, \Fcal_{j,m}) \otimes_\Ko \Kt  \simeq H^1(\KX, \Fcal_{j,m}) / \pi H^1(\KX, \Fcal_{j,m}) \hookrightarrow   H^1(\Xs, \Fcal_{j,m}|_{\Xs}).
\]
By Proposition \ref{cor cohomology},  we have 
\[
h^1(\Xs, \Fcal_{j,m}|_{\Xs})  \leq \frac{(m+j)^n}{n!} \biggl( \sum_{i\in I} b_i 
\widehat{h}^1 \Bigl( {E_i,\,} \Ocal\Bigl(\frac{m}{m+j}D+\frac{j}{m+j}M\Bigr)\Big|_{E_i}\Bigr) \biggr) + 
{o((m+j)^n)}.
\]
For the cycle ${\rm cyc}(\KX_s)$ associated to $\KX_s$, we have ${\rm cyc}(\KX_s)= \sum_{i \in I} b_i E_i$. 
Now the holomorphic Morse inequalities in Theorem \ref{theo holomorphic morse by Lazarsfeld}
applied on every component $E_i$  
and the above inequality {show} that 
{$h^1(\Xs, \Fcal_{j,m}|_{\Xs})$} is bounded above by 
\[
\frac{(m+j)^n}{n!} \left(n\Bigl(\frac{m}{m+j}D+\frac{j}{m+j}M_1\Bigr)^{n-1} \cdot \frac{j}{m+j}M_2  \cdot {\rm cyc}(\KX_s)     \right)  + o((m+j)^n).
\]
By flatness of $\KX$ over $\Ko$, the degrees of the special fibre  $\KX_s$ and 
the generic fibre $X$ of $\KX$ with respect to $n$ line bundles on $\KX$ are equal (cf. \cite[Prop.~2.10]{Kol96}). 
Hence the above upper bound is equal to 
\[
\frac{m^n}{n!} n  \Bigl\{D+ \frac{j}{m}M_1\Bigr\}^{ n-1} \cdot 
\Bigl\{\frac{j}{m}M_2\Bigr\} + 
o((m+j)^n) 
\leq \frac{m^n}{n!} n \{D+M_1\}^{n-1}\cdot \{M_2\} + {o(m^n)}
\]
using that $D,M_1,M_2$ are nef and $j \leq m$. 
This proves the claim.
\end{proof}

\begin{cor}
\label{cor first}
Let $\pi$ be a uniformizer of $\Ko$, let $D,M_1,M_2$ be nef divisors and let $\Ncal$ be any line bundle on $\KX$.
{There exists a function $\rho\colon \N\to \R$ with 
$\rho(m)=o(m^n)$ as $m\to \infty$ such that {for all $a \in \N$,} $m\in\N$ and all $j\in\{1,\ldots,m\}$,}
 we get 
\[ 
\length\left(  
H^1\big(\KX, \Ncal(mD +j(M_1-M_2))\big)_{\pi^a-{\rm tors}} \right) \leq 
\frac{m^n}{n!} a n \{D+M_1\}^{n-1}\cdot \{M_2\} + {a}\rho(m).
\]
\end{cor}

\begin{proof}
Since $\KX$ is projective, $ H^1\big(\KX, \Ncal(mD +j(M_1-M_2))\big)$ is a finitely generated $\Ko$-module.
Since $
\length (M_{\pi^a \text{-tors}}) \leq a \dim_\Kt(M\otimes_\Ko \Kt)
$
holds for any finitely generated $\Ko$-module $M$, the claim follows from Lemma \ref{lemma first}. \end{proof}

\subsection{Bounds for the zeroth cohomology group} 
\label{regularity free}

We continue working with the setup from the beginning of the chapter.
Let $E$ be an effective vertical Cartier divisor on $\KX$ and $s$ the
canonical global section of {$\Ocal(E)$}.
We write the Weil divisor {corresponding to $E$} as 
$\sum_{i \in I} c_i E_i$. We define $\alpha_i  \coloneqq c_i / b_i$ and 
$\alpha  \coloneqq \max_{i \in I} \alpha_i$.
 
Let $D,M_1,M_2$ be nef Cartier divisors on $\KX$. We consider the sum
\begin{equation}
\label{eq deltaD}
\delta_D(M_1,M_2) = \sum_{a,b,c} \{D\}^a \cdot \{M_1\}^b \cdot \{M_2\}^c
\end{equation}
of intersection numbers on $X$, where $(a,b,c) \in \N^3$ with
$a+b+c=n$ and $a \neq n$. By \cite[Prop.~2.10]{Kol96} we have
that 
\[\delta_D(M_1,M_2) = \sum_{a,b,c} D^a \cdot M_1^b \cdot M_2^c\cdot
  {\rm cyc}(\KX_s).\] 
This is non-negative  and will be used in the error terms of 
asymptotic estimates. Note {that the definition of
  $\delta_D(M_1,M_2)$ can be extended to the case when $M_{1}$ and
  $M_{2}$ are $\Q$-divisors and} 
\begin{equation} \label{degree term and epsilon}
\delta_D(\varepsilon M_1, \varepsilon M_2) = O(\varepsilon)
\end{equation}
for $\varepsilon \to 0$ in $\Q_{\geq 0}$. 
Let further $\Ncal$ be an arbitrary line bundle on $\KX$.

\begin{lemma}
\label{lemma zeroth'} \label{lemma zeroth}
There is an explicit constant $C_n>0$ depending only on $n$ such 
that for all $X,\KX,D,E, M_1, M_2, \Ncal$ as above,   
there exists a function $\rho\colon \N\to \R$ with 
$\rho(m)=o(1)$ as $m\to \infty$ such that for all $m\in\N$ and all $j\in\{0,\ldots, m\}$ we have 
\[
\left| \frac{n!}{m^n}h^0\bigl( E   , \Ncal(mD + j(M_1-M_2))|_{E} \bigr)  - D^n \cdot E \right| \leq C_n\delta_D(M_1,  M_2)\alpha + {\rho(m)}.
\]
\end{lemma}

\begin{proof}
We argue similarly as in the proof of Lemma \ref{lemma first}. 
For all $q \geq 0$, it follows from  
Proposition \ref{cor cohomology} and the holomorphic Morse inequalities 
\ref{theo holomorphic morse by Lazarsfeld}
that 
\begin{equation} \label{holomorphic Morse inequalities for coherent sheaf}
h^q( E, \Ncal(mD + j(M_1-M_2))|_{E} )\leq \frac{m^n}{n!} \binom{n}{q} 
\Bigl(D+\frac{j}{m}M_1\Bigr)^{n-q} \cdot \Bigl(\frac{j}{m} M_2\Bigr)^q \cdot E +
{\tilde\rho(m+j)}
\end{equation}
for some function 
$\tilde\rho\colon \N\to\R$ with $\tilde\rho(m)=o(m^n)$
as $m\to\infty$.
Using that $D,M_1,M_2$ are nef and using that the Weil divisor 
{${\rm cyc}(E)$ associated to $E$ satisfies 
${\rm cyc}(E)\leq \alpha \cdot {\rm cyc}(\KX_s)$},
we may replace $E$ in the bound \eqref{holomorphic Morse inequalities for coherent sheaf} by  $\alpha \cdot {\rm cyc}(\KX_s)$. 
As before, since the model $\KX$ is flat, the degree of the special fibre $\KX_s$ with respect to line bundles on $\KX$ agrees 
with the corresponding degree of the generic fibre $X$. For all $q\geq
1$, we deduce from \eqref{holomorphic Morse inequalities for coherent
  sheaf} and $j/m \leq 1$ that there is an explicit constant
  $C'_n$ depending only on $n$ such that   
\begin{equation} \label{error term for coherent sheaf}
 h^q\bigl(  E ,  \Ncal(mD + j(M_1-M_2))|_{E} \bigr) \leq
{\frac{m^n}{n!}}
  \alpha C'_n\delta_D(M_1,  M_2) 
 +\rho'(m)
\end{equation}
holds for all $m\in\N$ and $j\in\{1,\ldots,m\}$ with $\rho'(m)   \coloneqq \max\{
\tilde\rho(m+i)\mid 1\leq i\leq m\}$.

{By Proposition \ref{prop asymptotic chi},} the Euler characteristic $\chi(  E , \Ncal(mD + j(M_1-M_2))|_{E}  )$ equals
\begin{equation} \label{asymptotic RR for coherent sheaf}
\frac{m^n}{n!}  \sum_{q=0}^n (-1)^q \binom{n}{q}  \Bigl(D+\frac{j}{m}M_1\Bigr)^{n-q}  \cdot
\Bigl(\frac{j}{m} M_2\Bigr)^q \cdot E + O(m^{n-1}).
\end{equation}
Expanding \eqref{asymptotic RR for coherent sheaf}, 
bounding all terms involving at least one $M_i$ by 
$ C''_n\delta_D(M_1,  M_2)\alpha$ as above,  
using again $\cyc(E) \leq \alpha \cdot \cyc(\KX_s)$ and 
\eqref{error term for coherent sheaf}, we get the claim.
\end{proof}

\subsection{A filtration argument}
\label{subsection cutting}
 We consider a projective normal variety $X$ over $K$ with a projective normal model $\KX$ over $\Ko$. 
Let $f$ be a $\Z$-model function determined on $\KX$ by a vertical 
Cartier divisor $V \in {\rm Div}_0(\KX)$.   
In this situation we will write $\Ocal(f)  \coloneqq \Ocal(V)$.  

Since $\KX$ is projective, 
we can write $\Ocal(f) = \Ocal(M_1-M_2)$ for nef Cartier divisors $M_1,M_2$ on $\KX$. 
We consider a nef Cartier divisor $D$ on $\KX$ and we will use again
$\delta_D(M_1,M_2)$ from \ref{regularity free} to bound error terms.

In the following result, we assume $f \leq 0$. Then 
Proposition \ref{effective model fct} yields that the Cartier divisor
$E \coloneqq -V$ is effective 
and we denote the canonical global section of
$\Ocal(E)$ by $s$. 
We consider also an arbitrary line bundle $\Ncal$ on $\KX$.

\begin{lemma}
\label{lemma inductive step}
There is an explicit constant $C_n >0$ depending only on $n$ such that for every $X,\KX,D,f \leq 0, M_1, M_2, \Ncal$ as above 
{there exists a function $\rho\colon \N\to \R$ with $\rho(m)=o(1)$ as
$m\to\infty$ such that
\[
\left| \frac{n!}{m^n} {\length} \left( \frac {\Gamma(\KX,\Fcal_{j+1,m} )} {\Gamma(\KX,\Fcal_{j,m} ) } \right) - \int_\Xan f c_1(\Ocal(D))^{\wedge n} \right| 
\leq C_n \delta_D(M_1,  M_2) \cdot  \lceil |f|_{\rm sup} \rceil +\rho(m)
\]
holds for all $m \in \N$ and all $j\in\{0,\ldots,m-1\}$
where $\Fcal_{j,m} \coloneqq \Ncal(mD +j(M_1-M_2))$.}  
\end{lemma}

\begin{proof}
Recall that $\int_\Xan f c_1(\Ocal(D))^{\wedge n}$ was introduced in
\S \ref{prelim-monge-ampere}. By Lemma \ref{lem:equalityofmult},  we
have 
 \begin{equation} \label{CL and intersection numbers}
\int_\Xan (-f) c_1(\Ocal(D))^{\wedge n} = D^n \cdot E.
\end{equation}
The section $s$ determines a short exact sequence of coherent sheaves on $\KX$:
\begin{equation}
\label{equation short exact sequence s}
0 \longrightarrow \Fcal_{j+1,m} \stackrel{\otimes s}{\longrightarrow}
\Fcal_{j,m} \longrightarrow
\Fcal_{j,m}|_{E} \longrightarrow 0
\end{equation}
The associated long exact sequence in cohomology is 
\begin{multline}
\label{equation long exact sequence s}
0 \to \Gamma( \KX, \Fcal_{j+1,m}) \xrightarrow[]{\otimes s} 
\Gamma(\KX, \Fcal_{j,m} )  \xrightarrow[]{\phi_j}
\Gamma (E, \Fcal_{j,m})  \xrightarrow[]{\psi_j}
H^1( \KX, \Fcal_{j+1,m} ) \to \ldots .
\end{multline}
We have to compute 
$\length ( \im(\phi_j)) = \length \left(
  \Gamma(\KX,\Fcal_{j,m}) /   \Gamma(\KX,\Fcal_{j+1,m}) 
  \right)$.  
Using the obvious relation 
$
\length ( \Gamma (E, \Fcal_{j,m}))   = 
\length( \ker(\psi_j) ) + \length ( \im(\psi_j)) 
$ and  $ \im(\phi_j) = \ker(\psi_j)$, 
we deduce that  
\begin{equation} \label{length identity}
\length ( \im(\phi_j) )= \length ( \Gamma (E, \Fcal_{j,m})  ) 
-  \length( \im(\psi_j)).
\end{equation}
Using the notation from \S \ref{regularity free}, we have  $\alpha_i ={-} f(x_i)$,  hence Lemma \ref{lemma zeroth}  and \eqref{CL and intersection numbers} give 
\begin{equation} \label{first term bound}
 \left|\frac{n!}{m^n} \length ( \Gamma (E, \Fcal_{j,m})) - \int_\Xan
   (-f) c_1(\Ocal(D))^{\wedge n} \right| \leq C_n \delta_D(M_1,M_2)
 \cdot   |f|_{\rm sup} + \rho(m).
\end{equation}
For  $a  \coloneqq \lceil |f|_{\rm sup} \rceil$, the model function associated to the Cartier divisor $\div(\pi^a) {-} E$ equals $a + f \geq 0$ and 
hence Proposition \ref{effective model fct} shows that 
$\div(\pi^a) -E$ is an effective Cartier divisor on $\KX$. 
We deduce that $\KO_{E}$ is $\pi^{a}$-torsion and thus  
\[\im(\psi_j) \subset  H^1(\KX, \Fcal_{j+1,m})_{\pi^a\text{-tors}}.\]
This allows us to bound $\length(\im(\psi_j))$ using Corollary
\ref{cor first}. With \eqref{length identity} and \eqref{first term
  bound}, we get
\begin{equation} \label{conclusion length id}
\left|\frac{n!}{m^n} \length(\im(\phi_j))- \int_\Xan (-f)
  c_1(\Ocal(D))^{\wedge n} \right| \leq C_n \delta_D(M_1,M_2) \cdot
a + \rho(m) 
\end{equation} 
for larger $C_n$ and $ \rho$. By $\length ( \im(\phi_j)) =  \length \left(
  \Gamma(\KX,\Fcal_{j,m}) /   \Gamma(\KX,\Fcal_{j+1,m}) 
  \right)$, we get the claim. 
\end{proof}

\subsection{From model metrics to continuous semipositive metrics}
\label{subsection from formal}
In this subsection, $X$ is a normal {projective }variety of dimension $n$ 
over $K$ with a line bundle $L$. 
We will generalize 
the result from \S \ref{subsection cutting} to a continuous semipositive metric $\metr$ on $\Lan$ (cf. \S  \ref{closed-forms-pos-metrics}). 
Let $\overline{L} = (L, \metr)$ 
be the corresponding {metrized} line bundle. We will use the notation 
\[
\metr_{g} \coloneqq e^{-g}\metr 
\]
for any continuous function $g\colon \Xan \to \R$. 
If $f$ is a {$\Z$-}model function, if $\KL$ is a model of $L$ 
and if $\overline{L} = (L, \metr_{\KL})$,  
then  {$\metr_{\KL,f} = \metr_{\KL(f)}$ for 
$\KL(f)=\KL\otimes \Ocal(f)$}. 

\begin{art} \label{setup for nef decomposition}
 Let $\KX$ be a projective  $\Ko$-model of $X$ and 
let $f$ be a model function on $\Xan$ 
determined {on $\KX$.
Choose some non-zero $k\in\N$ such that $kf$ is a $\mathbb{Z}$-model
function determined on $\KX$. 
Similarly as before, there is a decomposition 
$\Ocal(kf)=\Ocal(kM_1-kM_2)$ for nef $\Q$-Cartier divisors 
$M_1,M_2$ on $\KX$ such that $kM_1,kM_2$ belong to 
$\textrm{Div}_0(\KX)$}. 

Since $\metr$ is a {continuous} semipositive metric on $\Lan$, it follows from \cite[Lemma 1.2]{BFJ1} that $L$ is nef.
Using algebraic intersection numbers on $X$, we have 
$$\delta_L(M_1,M_2)  \coloneqq \sum_{a,b,c} L^a \cdot \{M_1\}^b \cdot \{M_2\}^c \geq 0,$$
where $(a,b,c)$ ranges over $\N^3$ with $a+b+c=n$ and $a \neq n$.
{Note that in the setup of \eqref{eq deltaD}, we have  
$\delta_D(M_1,M_2) = \delta_L(M_1,M_2)$ for $L \coloneqq \mathcal{O}(D)|_{X}$.}
\end{art}

\begin{prop}
\label{prop from formal}
There is an explicit constant $C_n$ only depending on $n$ such that 
for all $X,L,f,M_1,M_2$ as above and  any continuous 
semipositive metric $\metr$ on $\Lan$, 
we have
$$
\left|\vol(L,\metr_{f}, \metr) - \int_\Xan f c_1(L,\metr)^{\wedge n}\right| \leq C_n\delta_L(M_1,  M_2)|f|_{\rm sup}.
$$
\end{prop}

\begin{proof}
We first prove the claim under the assumption that $f {\leq} 0$ 
and that $\metr$ is a semipositive model metric. 
We will proceed similarly as in the proof of Theorem \ref{cor3 energy}.  
We first choose a non-zero $k \in \N$ such that $kf$ is a $\Z$-model
function with $|kf|_{\rm sup} \in \N$, the divisors $kM_1,kM_2$ are Cartier divisors on $\KX$ and  $\metr^{\otimes k}$ is an algebraic metric. 
As we may always pass to a finer model (which does not change the quantities involved), we may assume that {$\metr^{\otimes k}=\metr_\KL$} for a nef line bundle {$\KL$} on $\KX$ {with $\KL|_X=L^{\otimes k}$.}  
We fix some $r \in \{0,\ldots,k-1\}$ and we consider the arithmetic progression $(m= kq +r)_{q \in \N}$. 
By passing to a finer model, we may assume that  $L^{\otimes r}$ has a model
$\Mcal$ on $\KX$ and that $\KX$ is normal.  
Similarly as in the proof of  Theorem \ref{cor3 energy}, we deduce from Lemma \ref{lemma twisted volume} that
\begin{align*}
&\length \left( \frac{\Hhat(X,L^{\otimes m}, \metr_{f}^{\otimes m})}{\Hhat(X,L^{\otimes m}, \metr^{\otimes m})} \right) 
= \length \left( \frac{\Hhat(X,L^{\otimes r} \otimes L^{\otimes {kq}}, \metr_\Mcal \otimes \| \ \|_{\KL(kf)}^{\otimes q})} {\Hhat(X,L^{\otimes r} \otimes L^{\otimes {kq}},\metr_\Mcal \otimes \metr_{\KL}^{\otimes q})} \right) + O(q^n),
\end{align*}
along the arithmetic progression $(m=kq +r)_{q\in \N}$. 

By Lemma \ref{sections and metrics}, the first summand on the right
hand side is equal to
\begin{equation} \label{filtration argument}
 \length \left( \frac{\Gamma(\KX,\Mcal \otimes {\KL(kf)}^{\otimes
       q})}{\Gamma(\KX,\Mcal \otimes {\KL}^{\otimes q})} \right) =  
{\length} \left( \frac {\Gamma(\KX,\Fcal_{q,q} )} {\Gamma(\KX,\Fcal_{0,q} ) }  \right) = 
\sum_{j=0}^{q-1}  {\length} \left( \frac {\Gamma(\KX,\Fcal_{j+1,q} )} {\Gamma(\KX,\Fcal_{j,q} ) }  \right)
\end{equation}
for any decreasing filtration $\Mcal \otimes {\KL}^{\otimes q}= \Fcal_{0,q} \supset \Fcal_{1,q} \supset \dots \supset \Fcal_{q,q}=\Mcal \otimes {\KL(kf)}^{\otimes q}$ into coherent $\Ocal_\KX$-submodules $\Fcal_{j,q}$ of $\Mcal \otimes {\KL}^{\otimes q}$. We will now apply Lemma \ref{lemma inductive step} with $q,\KL, kf, kM_1, kM_2, \Mcal$ instead of $m,\Ocal(D),f, M_1, M_2, \Ncal$ and hence we use the filtration $\Fcal_{j,q}:= \Mcal \otimes 
\KL^{\otimes q} \otimes \Ocal(j(kM_1-kM_2))$. Then Lemma \ref{lemma inductive step} shows that
\begin{equation} \label{application of filtration lemma}
\left| \frac{n!}{q^n} {\length} \left( \frac {\Gamma(\KX,\Fcal_{{j+1},q} )} {\Gamma(\KX,\Fcal_{{j},q} ) } \right) {-} \int_\Xan kf c_1(\KL)^{\wedge n} \right| 
\leq C_n \delta_\KL(kM_1,  kM_2) \cdot   |kf|_{\rm sup} +o(1).
\end{equation}
Now the claim in the special case can be deduced easily from \eqref{filtration argument} and \eqref{application of filtration lemma}.

Next, we skip the above assumption $f \leq 0$. 
Note that $C\coloneqq |f|_{\rm sup} \in \Q$ and hence $C$ is the model
function of a numerically trivial $\Q$-Cartier divisor $E{_{1}}$ on $\KX$ . 
The $\Q$-Cartier divisor $M_1'  \coloneqq M_1 {-} E{_{1}}$ is  nef. 
Replacing $k$ by a suitable multiple, we may assume that $kM_1'$ is also a Cartier divisor on $\KX$.
The decomposition
$\Ocal(k(f - C))=\Ocal(kM_1'-kM_2)$ follows from \ref{setup for nef decomposition}.  
An application of the  above special case to $f - C \leq 0$ gives  
$$\left| \vol(L,\| \ \|_{(f - C)}, \| \ \|)    -
\int_\Xan (f - C) c_1(L, \| \ \|)^{\wedge n} \right| \leq
C_n\delta_L(M'_1,  M_2)|f - C|_{\rm sup}. 
$$ 
We have $\vol(L,\| \ \|_{(f {-} C)}, \| \ \|)=\vol(L,\| \ \|_{f}, \| \ \| ) {-} C L^n$ by Remark \ref{remark volume nef} and 
 by Lemma \ref{volume of proportional metrics}. 
Now \ref{chambert loir measure for metrized line bundles}, $\delta_L(M_1,  M_2) = \delta_L(M'_1,  M_2)$ and $|f{-}C|_{\sup} \leq 2 |f|_{\sup}$ yield
$$ \left| \vol(L,\| \ \|_{f}, \| \ \| ) - 
\int_\Xan f c_1(L, \| \ \|)^{\wedge n} \right| \leq 2 C_n\delta_L(M_1,  M_2)|f|_{\rm sup}. $$
This proves the claim for a semipositive model metric.

Finally, we prove the claim for any continuous semipositive metric $\metr$. By definition, $\metr$ is a uniform limit of semipositive model metrics on $\Lan$ and hence the claim follows from continuity of the non-archimedean volume in Proposition \ref{lemma continuity non-archimedean volume} and of the Chambert--Loir measure in \ref{chambert loir measure for metrized line bundles}.
\end{proof}

\begin{theo}
\label{corollary differentiability 3}
Let $\| \ \|$ be a continuous semipositive metric on $\Lan$ and let {$f$ be a continuous function} on $\Xan$. 
Then if we consider everything fixed except $\varepsilon \in \R $, one has 
\begin{equation} 
\label{eq differentiablity 1}
\vol(L,\| \ \|_{\varepsilon f}, \| \ \| ) {=} \\
\varepsilon  \int_\Xan f c_1(L,\| \ \|)^{\wedge n} + {o(\varepsilon)}
\end{equation} 
for $\varepsilon \to 0$. In the special case of a model function $f$ on $\Xan$, the formula \eqref{eq differentiablity 1} holds even after replacing $o(\varepsilon)$ by $O(\varepsilon^2)$.
\end{theo}

\begin{proof}
It is enough to prove the claim for $\varepsilon > 0$. In the following, all $\varepsilon$ are assumed to be positive.
{We choose the same setup as in \ref{setup for nef decomposition}. 
For  $\varepsilon \in \Q_{>0}$, Proposition \ref{prop from formal} yields 
\begin{equation}
\label{equation epsilon rationnel}
\left|\vol(L,\metr_{\varepsilon f}, \metr) - \varepsilon \int_\Xan f c_1(L,\metr)^{\wedge n}\right| \leq C_n\delta_L(\varepsilon M_1, \varepsilon M_2)|\varepsilon f|_{\rm sup}.
\end{equation}
Using Proposition  \ref{lemma continuity non-archimedean volume}, this inequality and also  $\delta_L(\varepsilon M_1, \varepsilon M_2) = O(\varepsilon)$ from \eqref{degree term and epsilon} can be continuously extended to all $\varepsilon \in \R_{>0}$ and hence \eqref{eq differentiablity 1} follows for model functions.  }

To prove the case of a continuous function $f$, we argue by contradiction. Then either 
\begin{equation} \label{liminf contradiction}
\liminf_{\varepsilon \to 0} \frac{1}{\varepsilon}\vol(L,\| \ \|_{\varepsilon f}, \| \ \| ) <  \int_\Xan f c_1(L,\| \ \|)^{\wedge n}
\end{equation} 
or a reverse strict inequality with the $\limsup$ holds. We will prove that \eqref{liminf contradiction} leads to a contradiction, the case of the $\limsup$ is similar.

Let $\delta >0$. 
By density of model functions \cite[Thm.~7.12]{gubler-crelle}, there is a model function $f_\delta$ with $f-\delta \leq f_\delta \leq  f$. 
By \eqref{liminf contradiction}, we can choose $\delta > 0$ so small that 
\[
\liminf_{\varepsilon \to 0} \frac{1}{\varepsilon}\vol(L,\| \ \|_{\varepsilon f}, \| \ \| ) < \int_\Xan (f - \delta) c_1(L,\| \ \|)^{\wedge n} 
\leq 
  \int_\Xan f_\delta c_1(L,\| \ \|)^{\wedge n}.
  \]
By the model case, the right hand side equals $\liminf_{\varepsilon \to 0} {\varepsilon}^{-1}\vol(L,\| \ \|_{\varepsilon f_\delta},\metr)$. 
This contradicts the monotonicity of the volume as we have $\metr_{\varepsilon f} \leq \metr_{\varepsilon f_\delta}$ using $\varepsilon >0$.
 \end{proof}

\begin{rem}
\label{remark differentiability inequality}
We note here that only the use of the holomorphic Morse inequalities from Theorem \ref{theo holomorphic morse by Lazarsfeld} and our considerations about the asymptotic growth of algebraic volumes in Section \ref{section asymptotic}, applied in the proofs of Lemmas \ref{lemma first} and \ref{lemma zeroth}, allowed us to prove equality in \eqref{eq differentiablity 1}. 
Without using the holomorphic Morse inequalities, we can still prove ``$\geq$'' in  \eqref{eq differentiablity 1} as we explain below. This would have been enough for our applications to orthogonality in Section \ref{section orthogonality} and for the proof of Theorem \ref{theo intro monge ampere}. 
\end{rem}

The following result is a non-archimedean analogue of the main result in Yuan's paper \cite[Thm.~2.2]{Yua08}. It makes the lower bound in Proposition \ref{prop from formal} very explicit and leads to ``$\geq$'' in  \eqref{eq differentiablity 1} with the same arguments as in the proof of Theorem \ref{corollary differentiability 3}.

\begin{prop} 
\label{theorem inequality 1}
Let $f$ be a model function on $\Xan$ with $f \leq 0$ and let $\metr$ be a continuous semipositive metric on $\Lan$. 
Then  $f = - \log(\metr_1/\metr_2)$ for semipositive model metrics $\metr_1,\metr_2$ of a line bundle $M$ on $X$.
 For any such presentation, we have 
\begin{equation}
\label{eq theo inequality 1}
\vol(L,\metr e^{-f}, \metr) \geq   \int_\Xan f 
\left( c_1(L,\metr) + c_1(M,\metr_1) \right)^{\wedge n}.      
\end{equation}
\end{prop}

\begin{proof}
The existence of the presentation is equivalent to a decomposition  $\Ocal(kf)=\Ocal(kM_1-kM_2)$ as in \ref{setup for nef decomposition} and so the existence follows from \ref{setup for nef decomposition}.

To prove \eqref{eq theo inequality 1}, 
we need to review some of the results of this section. 
Under the same assumptions as in Lemma \ref{lemma zeroth}, we get 
the explicit upper bound 
\begin{equation} \label{explicit zeroth lower bound}
\frac{n!}{m^n}h^0\bigl( E   , \Ncal(mD + j(M_1-M_2))|_{E} \bigr)  \leq  (D+M_1)^n \cdot E  + o(1)
\end{equation}
by using the case $q=0$ in \eqref{holomorphic Morse inequalities for coherent sheaf}.
Observe that  
\eqref{holomorphic Morse inequalities for coherent sheaf}  
for $q=0$ is based only 
on the classical Hilbert--Samuel formula and on  
Proposition \ref{cor cohomology}. Under the assumptions and with the notation from Lemma \ref{lemma inductive step}, we get
\begin{equation} \label{explicit filtration lower bound}
\frac{n!}{m^n} {\length} \left( \frac {\Gamma(\KX,\Fcal_{j+1,m} )} {\Gamma(\KX,\Fcal_{j,m} ) }  \right)
\geq 
 \int_\Xan f c_1(\Ocal(D+M_1))^{\wedge n}  + o(1).
\end{equation}
Indeed, starting as in the proof of Lemma \ref{lemma inductive step} and  
using \eqref{length identity}, one gets that 
\begin{equation} \label{trivial lower bound}
 {\length} \left( \frac {\Gamma(\KX,\Fcal_{j,m} )} {\Gamma(\KX,\Fcal_{j+1,m} ) } \right)
\leq {\length} \left( \Gamma ( {E}, \Fcal_{j,m} ) \right) = h^0\bigl( E   , \Ncal(mD + j(M_1-M_2))|_{E} \bigr).
\end{equation} 
Applying \eqref{explicit zeroth lower bound}, we deduce that 
\[
\frac{n!}{m^n}
{\length} \left( \frac {\Gamma(\KX,\Fcal_{j,m} )} {\Gamma(\KX,\Fcal_{j+1,m} ) } \right)
\leq 
 (D + M_1)^{n}\cdot {E} + o(1)
=
  \int_\Xan (-f) c_1(\Ocal(D+M_1))^{\wedge n}  +o(1)
\]
where the last equality follows from Lemma \ref{lem:equalityofmult} applied to the $\Z$-model function $-f$ associated to $E$.  
Multiplying by $-1$, we get \eqref{explicit filtration lower bound}.

Now Proposition \ref{theorem inequality 1} follows from the same arguments as used in the proof of Proposition \ref{prop from formal} just by 
replacing  the application of Lemma \ref{lemma inductive step} in \eqref{application of filtration lemma} by \eqref{explicit filtration lower bound}.
 \end{proof}

\section{Application to orthogonality and  Monge--Amp\`ere equation}
\label{section orthogonality}

In this section  $K$ is a complete discretely
valued field with valuation ring $K^{\circ}$ and residue field
$\tilde K$. 
{At the end of Subsection \ref{subsecton orthogonality}} 
we will assume that ${\rm char}(\tilde{K})=0$.

\subsection{A local approach to semipositivity}\label{local approach to sp}

In this subsection, $L$ is a line bundle on a proper variety $X$ over $K$. It will be important to have a local analytic characterization of semipositive model metrics. This is 
done in \cite[\S 6]{gubler-kuennemann2} over an algebraically closed non-archimedean base field and can be done in a similar way over a complete discretely valued field (see \cite{gubler-martin} for details and generalizations). 
Our analytic objects will be compact strictly $K$-analytic domains $V$ \cite[p.\ 48]{berkovich-book} in the analytification $\Xan$ of $X$. We mimick the construction of algebraic metrics from \ref{models}. We consider now {\it formal models} 
$\fV$ of $V$ which are admissible formal schemes over $\kcirc$  \cite[\S 1]{BL1} with generic fiber $V$. Similarly as in \ref{models}, a formal model $(\fV,\fL)$ of $(V,\Lan|_V)$ induces a metric $\metr_\fL$ on $\Lan|_V$ which we call {\it the formal metric associated to $\fL$}.

Following \cite[6.2]{gubler-kuennemann2} and \cite{gubler-martin}, we say that a model metric $\metr$
on $L^{\an}$ is {\it semipositive in $x \in \Xan$} if there exist  $k \in \N \setminus \{0\}$, a compact
strictly $K$-analytic domain $V$ which is a
neighbourhood of $x$, and a formal model $(\fV,\fL)$ of $(V,(\Lan)^{\otimes k} |_V)$
with $\metr|_V^{\otimes k} = \metr_\fL$ such that for any  curve $Y$
in the special fibre of $\fV$, which is proper over $\tilde K$, we have $\deg_\fL(Y)\geq 0$. 
By \cite[6.5]{gubler-kuennemann2} and \cite[Prop. 3.10]{gubler-martin}, the model metric $\metr$ is
semipositive if and only if it is
semipositive in all $x \in \Xan$.

We will need the following result from \cite[Prop. 3.11]{gubler-martin}.
\begin{prop}
 \label{convexity and max for psh}
Let $\|\ \|_1$ and $\| \ \|_2$ be model metrics on $L^\an$.  
Then the metric $\| \ \| \coloneqq \min ( \| \ \|_1, \| \ \|_2)$ is a
model metric on 
$L$. If $\| \ \|_1$ and $\| \ \|_2$ are semipositive in $x \in
\Xan$, then  
$\| \ \|$ is semipositive in $x$.
\end{prop}

\begin{art}\label{definespecialsingularmetric-art}

Let $s_0 \in \Gamma(X, L) \setminus \{ 0 \}$. 
We define a singular metric $\| \ \|_{s_0}$ on $\Lan$ by

\begin{equation}\label{definespecialsingularmetric}
\| s\|_{s_0}(x) =\Bigg\{
\begin{array}{ccl}
\left| \frac{s}{s_0}(x) \right|&\mbox{if}&
\frac{s}{s_0}\in \Ocal_{\Xan,x},\\
\infty &\mbox{if}&\frac{s}{s_0} \notin \Ocal_{\Xan,x}.
\end{array}
\end{equation}

\end{art}

\begin{lemma}
\label{lemma semipositive section min}
Let $\metr$ be a model metric on $L^{\an}$ and
$s_0 \in \Gamma(X, L) \setminus \{ 0 \}$.
Let $\metr_{s_0}$ be the singular metric defined above. 
Then  $ \metr' \coloneqq \min \left(\metr, \metr_{s_0}
\right)$ is a model metric on $L^{\an}$. If $\metr$ is semipositive in
$x \in \Xan$, then $\metr'$ is also semipositive in $x$. 
\end{lemma}

\begin{proof}
By passing to a positive tensor power, we may assume that $\metr$ is an algebraic metric. 
It follows from \cite[Prop. 8.13]{gubler-kuennemann} that algebraic 
metrics and formal metrics on $L^{\an}$ are the same as the argument 
in {\it loc.~cit.~}does not use that the base field is algebraically closed.   
Thus, to prove the
first claim, it is enough to show that $\metr'$ is
a formal metric on $L^{\an}$. 
We use the fact that being a formal metric on $L^{\an}$ is a $G$-local
property (cf. \cite[Prop. 5.10]{gubler-kuennemann2} and \cite[Prop.~2.8]{gubler-martin}).
 By
\cite[Lemma 1.6.2]{Berko93}, it is enough to check that for any $y \in
\Xan$, there is a neighborhood $V$ which is a strictly affinoid
domain in $\Xan$ such that $\metr'$ restricts to a formal metric on
$V$.  

{Let us} first assume  $s_0(y) =0$. 
Since $\Xan$ is a good analytic space, there is  a  neighborhood $V$ of $y$  which is a strictly affinoid domain
in $\Xan$ and a frame $s$ of $L$ over $V$ which satisfies 
$\| s(v) \| < \| s(v) \|_{s_0}$ for all $v\in V$.
So $\metr'_{|V} = \metr_{|V}$ is a formal metric on
$L^{\an}|_V$.

If $s_0(y) \neq 0$, then we can find a neighbourhood $V$ of $y$ which
is a strictly affinoid domain in $\Xan$ such that 
${s_0}|_{V}$ is nowhere vanishing. 
So the restriction of $\metr_{s_0}$ to $V$ is isometric to the
trivial metric on $\Ocal_V$ which is formal.
Hence the restriction of $\metr'$ to $V$ is the minimum of two
formal metrics on $V$.
By \cite[Lemma 7.8]{gubler-crelle}, the restriction of 
$\metr'$ to $V$ is also a formal metric on  $\Lan$. This proves the
first claim.

If $\metr$ is semipositive in $x$, then we proceed as in the first
part of the proof with $y  \coloneqq x$ to show that $\metr'$ is semipositive in
$x$. If $s_0(x)=0$, then this follows from the fact that 
$\metr'_{|V} = \metr_{|V}$ is semipositive in $x$. If $s_0(x)\neq 0$
and $V$ is as before, then \cite[Cor. 5.12]{gubler-kuennemann2}
and \cite[Prop. 2.6]{gubler-martin}
give the existence of an algebraic metric on $L^{\an}$ which agrees with
the singular metric $\metr_{s_0}$ over $V$. Since  
$\metr'_{|V}$ is the restriction of the minimum of two model
metrics on $L^{\an}$ which are both semipositive on $V$,  Proposition
\ref{convexity and max for psh} yields that $\metr'$ is semipositive
on $V$. 
\end{proof}

\subsection{A useful property of the semipositive envelope of a metric}
\label{subsection useful property}

Let $X$ be a normal projective $K$-variety. 
Let $L$ be a line bundle on $X$ and
$\metr$ a continuous metric {on} $L^\an$. 
We will assume that the semipositive envelope $P(\metr)$ is a continuous metric. 
If $\cha(\tilde K)=0$ and if $L$ is an ample line bundle on a projective smooth variety, then the semipositive envelope $P(\metr)$ of $\metr$
is a continuous metric on $L^\an$ 
(see Theorem \ref{prelim-semipos-envelope}). 
Going from a continuous metric to its semipositive
envelope does not change the space of small sections as we will show next.

\begin{prop}
\label{prop vol envelope} 
For a continuous metric $\metr$ on the line bundle $\Lan$
such that the semipositive envelope $P(\metr)$ 
is a continuous metric, we have
\begin{equation}
\label{eq equality bounded section} 
\hH^0(X,L,\metr) = \hH^0(X,L,P(\metr)).
\end{equation}
As a consequence, the non-archimedean volume satisfies
\begin{equation}
\label{eq equality volume}
    \vol(\metr,P(\metr))=0.
\end{equation}
\end{prop}

\begin{proof}
Let us first prove \eqref{eq equality bounded section}.
We have $\|s\|\leq P(\|s\|)$
for every section $s\in \Gamma(X,L)$ 
by definition of the semipositive envelope.
This implies  
$\hH^0(X,L,P(\metr)) \subseteq \hH^0(X,L,\metr) $.
Assume that there exists some $s_0\in \hH^0(X,L,\metr)$
which does not belong to the subset $\hH^0(X,L,P(\metr))$.
Then $\|s_{0}\|\le 1$ and there is a point $x_0\in X^{\an}$
with
\begin{equation}\label{eq:2-new}
P(\|s_{0}(x_0)\|)> 1.
\end{equation}
This gives $f  \coloneqq \log \|s_0\|\leq 0$ and the metric 
$\metr_{s_0}=\metr e^{-f}$ introduced in \ref{definespecialsingularmetric-art}
satisfies $\metr\leq \metr_{s_0}$.
For a semipositive model metric $\metr_1 \geq \metr$ on $L^{\an}$, we get
\begin{equation}\label{prime-metric-ineq}
\metr\leq \metr'  \coloneqq \min(\metr_{s_0},\metr_1) \leq \metr_1.
\end{equation}
By Lemma \ref{lemma semipositive section min},  $\metr'$ is a semipositive model metric on $L^{\an}$.
Hence $P(\metr) \leq \metr'$ by \eqref{prime-metric-ineq} and the construction of 
the semipositive envelope. However we have $\|s_0\|_{s_0}=1$ and get
\begin{equation} \label{eq:metr'-new}
\|s_{0}(x)\|'=\min(1,\|s_{0}(x)\|_1) \leq 1
\end{equation}
for all $x \in \Xan$. 
This contradicts  $P(\metr) \leq \metr'$ if we compare \eqref{eq:2-new}
and \eqref{eq:metr'-new}.

Equation \eqref{eq equality volume} is a direct consequence of
\eqref{eq equality bounded section}  by definition
of the non-archimedean volume in \ref{defi non-archimedean volume}
and {Remark \ref{homogenity of envelope}}.
\end{proof}

\begin{cor}\label{cor 4 energy}
Let $L$ be a line bundle on $X$ and let $\| \ \|_1$ and $\| \ \|_2$ be continuous metrics on $L^{\an}$
{whose semipositive envelopes $P(\metr_1)$ 
and $P(\metr_2)$ are continuous metrics}.
Then we have
$\vol(L, \metr_1, \metr_2) = E(L,P(\metr_1),P(\metr_2))$  
and the $\limsup$ in the
definition of the non-archimedean volume is  a limit.
\end{cor}

\begin{proof}
For $i=1,2$, Proposition \ref{prop vol envelope} yields 
$$\Hhat(X,L,\metr_i)=\Hhat(X,L,P(\metr_i)), \quad\vol(L,\metr_1,\metr_2)=\vol(L,P(\metr_1),P(\metr_2)).$$
Hence the result follows from Theorem  \ref{cor3 energy} and Remark \ref{homogenity of envelope}. 
\end{proof}

\subsection{The orthogonality property}
\label{subsecton orthogonality}

Let $X$ be a normal  projective $K$-variety of dimension $n$.
After the proof of Theorem \ref{theo ortho new} 
we will assume that $\text{char}( \tilde K) =0$ 
which implies in particular that the semipositive 
envelope $P(\metr)$ of a continuous metric $\metr$
of an ample line bundle on a smooth projective variety over $K$  is a continuous metric by a result
of Boucksom, Favre, and Jonsson (see \ref{prelim-semipos-envelope}).

\begin{defi}
Let $L$ be a  line bundle on $X$. 
Let $\metr$ be a continuous metric on $L^{\an}$
{whose semipositive envelope $P(\metr)$ is continuous.}
We say that \emph{the pair $(L, \metr)$ satisfies the orthogonality property} if 
\begin{displaymath}
\int_{X^{\an}}\log \frac {P(\metr)} {\metr} c_1(L, P(\metr))^{\wedge n} =0.
\end{displaymath}
\end{defi}

\begin{theo}\label{theo ortho new} 
{Let $L$ be a line bundle  on $X$ and
$\metr$ a continuous metric on $L^{\an}$
whose semipositive envelope $P(\metr)$ is a continuous metric.
Then the pair $(L, \metr)$} satisfies the orthogonality property.  
\end{theo}

\begin{proof}
{By assumption the} function $\varphi = \log \frac {P(\metr)} {\metr}$ 
is continuous. 
Fix $\varepsilon\in [0,1]$.
We have $\metr \leq P(\metr)  e^{-\varepsilon \varphi} \leq P(\metr)$.
Hence $P(P(\metr)  e^{-\varepsilon \varphi}) = P(\metr)$. 
Applying  Proposition \ref{prop vol envelope} and then Theorem \ref{corollary differentiability 3}, we get
\begin{equation*}\label{vanisch-vol-envelope-epsilon}
0=\vol\bigl(P(\metr) e^{-\varepsilon \varphi}, P(\metr)\bigr) = \varepsilon \int_{X^{\an}} \varphi c_1(L, P(\metr))^{\wedge n}  + o(\varepsilon) 
\end{equation*}
for $\varepsilon \rightarrow 0$. Dividing first by $\varepsilon$ and then letting $\varepsilon \rightarrow 0$, we get the result.
\end{proof}

We now use the notations and terminology from \S \ref{closed-forms-pos-metrics}
and assume for the rest of this subsection 
that $\text{char}( \tilde K) =0$ and that $X$ is a smooth projective variety over $K$.  
Let $\theta \in \Zcal^{1,1}(X)$ be a closed $(1,1)$-form such that 
$\{\theta\}\in N^1(X)$ is ample.
Given $f\in C^{0}(X^{\an})$ 
we denote by $P_\theta(f)$ the $\theta$-psh envelope of $f$ defined in 
\cite[8.1]{BFJ1}
and by 
$\MA_\theta(\varphi)$ the Monge--Amp\`ere measure on $X^\an$
associated with a continuous $\theta$-psh function $\varphi$
\cite[Thm.~3.1]{BFJ2}.
The form $\theta$
is said to \emph{satisfy the orthogonality property} if 
\[
\int_{X^{\an}}(f-P_{\theta }(f))\MA_\theta(P_{\theta }(f))=0
\]
holds for all $f\in C^0(X^\an)$ \cite[Def. (A.1)]{BFJ2}. 
Boucksom, Favre and Jonsson show in \cite[App. A]{BFJ2} that every such $\theta$ satisfies the orthogonality property if 
$X$ satisfies the algebraicity condition $(\dagger)$ mentioned in \S \ref{subsec: MA}. Using our results, we can remove $(\dagger)$:

\begin{theo} \label{BFJ corollary 1}
Let $\theta \in \Zcal^{1,1}(X)$ be a closed form such that $\{\theta\}$ is ample. 
Then $\theta$ satisfies the orthogonality property. 
\end{theo}

\begin{proof}
To deduce this from Theorem \ref{theo ortho new}, we follow \cite{BFJ2}.
By \cite[Lemma A.2]{BFJ2} it is enough to show the theorem for 
rational classes. 
{Homogeneity} of the envelope allows to assume that $\theta$ is an integral class.
In this case the Monge--Amp\`ere measure $\MA_\theta(P_\theta(f))$ agrees 
with the Chambert-Loir measure $c_1(L, P(\metr))^{\wedge n}$ 
(see \cite[3.3]{BFJ2}).
Then the result follows from Theorem \ref{theo ortho new}. 
\end{proof}

Now we can solve the Monge--Amp\`ere problem  without the algebraicity 
assumption $(\dagger)$.
For the definition of the dual complex of an SNC model, see 
 \cite[\S 3]{BFJ1}.

\begin{cor} \label{BFJ corollary 2}
Let $\theta\in \Zcal^{1,1}(X)$ be a closed form with $\{\theta\}$  ample and $\mu$ a positive Radon measure on $X^\an$ of mass ${\{\theta\}^n}$.
If $\mu$ is supported on the dual complex of some SNC model of $X$
then there exists a continuous $\theta$-psh function $\varphi$
such that $\MA_\theta(\varphi)=\mu$.
\end{cor}

\begin{proof}
This follows from Theorem \ref{BFJ corollary 1} and \cite[Thm.~8.1]{BFJ2}.
\end{proof}

\begin{rem}   \label{remark diff eq or} By \cite[Rem.~7.4]{BFJ2}, the orthogonality property is equivalent 
to the differentiability of $E \circ P_\theta$. 
Note that our differentiability result in Theorem \ref{corollary differentiability 3} is a priori different and weaker.  We only proved for semipositive $\theta$ that 
the function 
$t\in \R \mapsto E \circ P_\theta(tf)$ is differentiable at $t=0$ for any $f\in C^0(X^{\an})$. 
However, the orthogonality property 
from Theorem \ref{BFJ corollary 1} 
and the proof of \cite[Cor.~7.3]{BFJ2} imply that 
$f \in C^0(X^{\an}) \mapsto E \circ P_\theta(f)$ is differentiable 
in the direction of any $g  \in C^0(X^{\an})$.
\end{rem}

\appendix

\section{Holomorphic Morse inequalities in arbitrary characteristic\\
by Robert Lazarsfeld}
\label{subsection_Morse}

\setcounter{theo}{0}

The holomorphic Morse inequalities give us asymptotic upper bounds for
the higher cohomology of powers of line bundles. They were first proved
by J.P.~Demailly \cite{Dem85} for complex varieties. 
Later F.~Angelini
\cite{Ang96} gave an algebraic proof for varieties over a field of
characteristic zero (see also \cite[Example 2.4]{Kur06}). In this
section,  we 
extend the holomorphic Morse inequalities to varieties over
arbitrary fields.

\begin{rem}
\label{remark general position}
We say that a
property (P) holds 
{\it at points in general position}
(resp.~{\it at points in very general position})
of an irreducible variety $T$ over a field $k$ if (P) holds 
on the complement of a proper Zariski closed subset of $T$ 
(resp.~on the complement of a countable union 
of  proper Zariski closed subsets of $T$).
If $k$ is uncountable and algebraically closed and (P) holds at
points in very general position, one can always pick a $k$-rational point 
where (P) holds
(this is not true if $k$ is only countable).
\end{rem}

We have introduced the space $\Div(Y)_\R$ of real Cartier divisors on a projective scheme $Y$ over $k$ in \S \ref{subsection volumes}. Such a divisor $D$ is called nef if the intersection number with any closed curve in $Y$ is non-negative. Now we come to the {\it holomorphic Morse inequalities}.

\begin{theo}
\label{theo holomorphic morse by Lazarsfeld}
Let $Y$ be an $n$-dimensional projective scheme  over any field $k$ and let $q\in \{0,\ldots,n\}$. For very ample Cartier divisors $D,E$ on $Y$ and $F\coloneqq D-E$, we have  
\begin{equation}
\label{eq HMI}
h^q(Y, \Ocal_Y(mF)) \leq \binom{n}{q}  D^{n-q} \cdot E^q  \frac{m^n}{n!} + O(m^{n-1}).
\end{equation} 
More generally, if $D,E \in \Div(Y)_\R$ are nef, then \eqref{eq HMI} holds with the weaker error term $o(m^n)$ for $m \to \infty$ instead of $O(m^{n-1})$.
\end{theo}

\begin{proof}
{\it Step 1: The claim holds for very ample Cartier divisors $D,E$ on a projective variety $Y$ over an algebraically closed field $k$.}

The numbers $h^q$ and the intersection numbers are invariant under base change (see \cite[III 9.3]{Hart} and \cite[Example
6.2.9]{Ful98}) and hence we may assume that the base  $k$ is  uncountable. We denote by $|E|$  the space of hyperplane sections of $E$.
According to \cite[Prop.~5.5]{Kur06},
for  fixed integers $m\geq 0$, $n\geq s \geq 0$ and $n\geq j \geq 0$, 
\begin{equation}
\label{eq Kur 5.5}
h^j(s,m)  \coloneqq h^j(E_1 \cap \ldots \cap E_s, \Ocal(mD))
\end{equation}
does not depend on the choice of 
 divisors  
$E_1,\ldots,E_s \in |E|$ in general position.
It follows that for divisors  $E_1, \ldots, E_s
\in |E|$ in very general position,   
the
equality \eqref{eq Kur 5.5} holds simultaneously for all $m\geq 0$, $n\geq s \geq 0$ and $n\geq j \geq 0$. 
Since we assume that $k$ is uncountable, such divisors exist.  
Since {$D$} is very ample,  there exists  $m_0 \in \N$ such that 
\begin{equation}
\label{eq Kur vanishing}
\text{$  
h^j(s,m)=0$ for all integers   $m\geq m_0, \ n\geq  j\geq 1,    n\geq s \geq 0 $.} 
\end{equation}
For a fixed integer $s $ with $n\geq s \geq 0$ and varying $m \in \N$, we claim that   
\begin{equation}
\label{eq asymptotic very general}
h^0(s,m) {=} D^{n-s} \cdot E^s \frac{m^{n-s}}{(n-s)!} + O(m^{n-s-1}).
\end{equation}
To see this, we note first that a Bertini-type argument shows that the intersection product  $E^s$ is given by the scheme theoretic intersection $E_1 \cap \ldots \cap  E_s$ (see \cite[Lemma 5.7]{Kur06}). Using that $D$ is very ample and Remark \ref{remark volume nef}, we deduce \eqref{eq asymptotic very general}.

Applying Lemma 5.7 and Corollary 4.2 of \cite{Kur06} for a fixed integer $m>n$,  we deduce that
for effective Cartier divisors
$(E_1,\ldots,E_m)\in |E|^m$ in general position we have the following exact sequence: 
\begin{multline}
\label{eq Kur 4.2}
0 \to \OY\Bigl(mD-\sum_{i=1}^m E_i \Bigr) \to \OY(mD) \to \bigoplus_{1\leq i \leq m} \Ocal_{E_i} (mD) \to 
\\
\bigoplus_{1\leq i_1 <i_2 \leq m} \Ocal_{E_{i_1}\cap E_{i_2}} (mD) \to \cdots  \to 
\bigoplus_{1\leq i_1 < i_2<\cdots<i_n \leq m} \Ocal_{E_{i_1}\cap E_{i_2} \cap \ldots \cap E_{i_n}} (mD) \to 0
\end{multline}

We fix now an integer $m \geq \max(n+1,m_0)$. There are $E_1, \dots, E_m \in |E|$  such that \eqref{eq Kur 4.2} is exact  and such that for 
 any integer $0\leq s\leq n$ and for any integers
$1\leq i_1 <\ldots <i_s\leq m$, the $s$-tuple 
$E_{i_1},\ldots,E_{i_s}$ is in very general position. 
The latter yields that $h^j(s,m) = h^j(E_{i_1} \cap \ldots \cap E_{i_s}, \Ocal(mD))$.  
We conclude from \eqref{eq Kur vanishing}   that \eqref{eq Kur 4.2} gives  
an acyclic resolution of the sheaf $ \OY(mD-\sum_{i=1}^m E_i ) \simeq \Ocal(mF)$. 
It follows that $H^q(Y,\OY(mF)) \simeq \ker(d^q) / \im(d^{q-1})$ for the canonical homomorphism  
$$d^q\colon \bigoplus_{|I|=q}  
 H^0(E_I, \Ocal_{E_I} (mD) ) \rightarrow 
\bigoplus_{|J|=q+1}  
 H^0(E_J, \Ocal_{E_J} (mD) ), $$
where $I,J$ ranges over  subsets of $\{1, \dots, m\}$ and  where $E_I  \coloneqq \bigcap_{i \in I} E_i$.  
We conclude 
$$h^q(Y,\OY(mF)) \leq \sum_{|I|=q} h^0(E_I, \Ocal_{E_I} (mD) )
= \binom{m}{q} h^0(q,m).$$
The first step follows now from \eqref{eq asymptotic very general} and  $\binom{m}{q} = \frac{m^q}{q!} +O(m^{q-1})$ for fixed $q$.

{\it Step 2. The inequalities \eqref{eq HMI} hold for very ample Cartier divisors $D,E$ on a projective scheme $Y$ over any  field $k$.}

By the same base change argument as in Step 1, we may assume that $k$ is algebraically closed. 
Let $[Y]=\sum_{i \in I} b_i Y_i$ be the fundamental cycle of the projective scheme $Y$, where $Y_i$ ranges over the irreducible components of $Y$ and where $b_i$ is the multiplicity of $Y$ in $Y_i$ given as the length of the local ring at the generic point of $Y_i$. The first step shows 
$${h}^q(Y_i, \Ocal_{Y_i}(mF)) \leq \binom{n}{q}  D^{n-q} \cdot E^q \cdot Y_i \frac{m^n}{n!} + O(m^{n-1})$$
and hence Lemma \ref{lemma bound cohomology non reduced} yields Step 2 by the following computation:
\begin{align*}
  &\widehat{h}^q(Y, \Ocal_Y(mF)) \le \sum_{i\in I}b_{i}h^{q}(Y_{i},\Ocal_{Y_i}(mF))+ O(m^{n-1})\\
  &\le \sum_{i\in I}b_{i}\binom{n}{q}
  D^{n-q} \cdot  E^{q} \cdot Y_i \frac{m^{n}}{n!}+ O(m^{n-1})
  \le \binom{n}{q}
 { D^{n-q} \cdot E^{q}}\frac{m^{n}}{n!}+ O(m^{n-1}),
\end{align*}

{\it Step 3. The case of nef real divisors $D,E$ on a projective scheme $Y$ over any  field $k$.}

 By  definition of asymptotic cohomological functions,  it is equivalent to prove 
  \begin{equation}\label{eq:9}
    \widehat{h}^q(Y,F)\le \binom{n}{q}  D^{n-q}\cdot E^q.
  \end{equation}
  It is here where the error term $o(m^n)$ comes in. Since both sides are
  continuous (see Proposition \ref{lemma volume real divisors'}) and the ample cone is
  dense inside the nef cone, we may assume that $D,E$ are ample $\Q$-Cartier divisors. 
	Since both sides
  of the equation are homogeneous of degree $n$
  (see Proposition \ref{lemma volume real divisors'}), we may assume that $D,E$ are very ample Cartier divisors on $Y$ and hence Step 3 follows from Step 2. 
\end{proof}

{\footnotesize
{\sc R. Lazarsfeld, Department of Mathematics, Stony Brook University, Stony Brook, NY 11794, USA}\\
\indent
{\it E-mail address:} {\tt robert.lazarsfeld@stonybrook.edu}
}

\bibliographystyle{alpha}
\def\cprime{$'$}

\end{document}